\documentclass{amsart}

\usepackage{amsfonts,amsmath}
\usepackage{amsthm,amsxtra,amssymb}
\usepackage[bookmarks,colorlinks,linkcolor=blue,citecolor=blue,urlcolor=blue]{hyperref}
\usepackage[all]{xy}
\usepackage{color}
\usepackage{stmaryrd}
\usepackage[vcentermath]{youngtab}
\usepackage{aliascnt}
\usepackage{tikz}
\usepackage{mathabx}

\usetikzlibrary{arrows,matrix,positioning,decorations.pathreplacing}
\tikzset{
v/.style={draw, fill, circle, minimum size=1.5mm, inner sep=0},
e/.style={very thick}
}

\DeclareMathOperator{\im}{im}
\DeclareMathOperator{\coker}{coker}
\DeclareMathOperator{\Hom}{Hom}
\DeclareMathOperator{\End}{End}
\providecommand{\id}{\ensuremath\mathrm{id}}

\DeclareMathOperator{\Aut}{Aut}
\DeclareMathOperator{\AutF}{AutF}

\DeclareMathOperator{\GL}{GL}

\DeclareMathOperator{\IA}{IA}
\DeclareMathOperator{\Mod}{Mod}

\DeclareMathOperator{\Sp}{Sp}

\DeclareMathOperator{\Fun}{Fun}
\DeclareMathOperator{\op}{op}

\DeclareMathOperator{\Lk}{Lk}

\DeclareMathOperator{\Ind}{Ind}
\DeclareMathOperator{\Res}{Res}

\DeclareMathOperator{\rk}{rk}
\DeclareMathOperator{\Fix}{Fix}

\DeclareMathOperator{\inc}{inc}

\DeclareMathOperator{\Lan}{Lan}
\DeclareMathOperator{\coeq}{coeq}

\DeclareMathOperator{\colim}{colim}
\providecommand{\FI}{\ensuremath\mathsf{FI}}

\providecommand{\VIC}{\ensuremath\mathsf{VIC}}
\providecommand{\SI}{\ensuremath\mathsf{SI}}
\providecommand{\A}{\ensuremath\mathcal{A}}
\providecommand{\C}{\ensuremath\mathcal{C}}
\providecommand{\D}{\ensuremath\mathcal{D}}
\providecommand{\G}{\ensuremath\mathcal{G}}

\providecommand{\Cmod}{\ensuremath{\mathcal{C}\mathsf{-mod}}}
\providecommand{\CSet}{\ensuremath{\mathcal{C}\mathsf{-Set}}}
\providecommand{\xmod}[1]{\ensuremath{#1\mathsf{-mod}}}

\providecommand{\Set}{\ensuremath \mathsf{Set}}
\providecommand{\hooklongrightarrow}{\lhook\joinrel\longrightarrow}
\providecommand{\twoheadlongrightarrow}{\relbar\joinrel\twoheadrightarrow}
\providecommand{\inject}{\hooklongrightarrow}
\providecommand{\surject}{\twoheadlongrightarrow}

\providecommand{\quot}[2]{{\raisebox{.2em}{$#1\!\!$}\left/\raisebox{-.2em}{$\!#2$}\right.}}
\newcommand{\doublequot}[3]{{\left.\raisebox{-.2em}{$#1\!$}\right\backslash\raisebox{.2em}{$\!#2\!$}\left/\raisebox{-.2em}{$\!#3$}\right.}}

\providecommand{\NN}{\ensuremath\mathbb N}
\providecommand{\ZZ}{\ensuremath\mathbb Z}
\providecommand{\QQ}{\ensuremath\mathbb Q}

\definecolor{grey}{gray}{.5}

\providecommand{\con}[1]{\textbf{\textup{#1}}}

\numberwithin{thmcounter}{section}
\newaliascnt{thmauto}{thmcounter}

\newaliascnt{Defauto}{thmcounter}

\newaliascnt{exauto}{thmcounter}

\newaliascnt{exsauto}{thmcounter}

\newaliascnt{lemauto}{thmcounter}

\newaliascnt{propauto}{thmcounter}

\newaliascnt{corauto}{thmcounter}

\newaliascnt{remauto}{thmcounter}

\theoremstyle{plain}
\newtheorem{thm}[thmauto]{Theorem}
\newtheorem{Def}[Defauto]{Definition}
\newtheorem{ex}[exauto]{Example}
\newtheorem{exs}[exsauto]{Examples}
\newtheorem{lem}[lemauto]{Lemma}
\newtheorem{prop}[propauto]{Proposition}
\newtheorem{cor}[corauto]{Corollary}
\newtheorem{rem}[remauto]{Remark}
\newtheorem{thmA}{Theorem}

\newtheorem*{rem*}{Remark}
\newtheorem*{thm*}{Theorem}
\newtheorem*{exs*}{Examples}

\numberwithin{equation}{section}

\let\originalleft\left
\let\originalright\right
\renewcommand{\left}{\mathopen{}\mathclose\bgroup\originalleft}
\renewcommand{\right}{\aftergroup\egroup\originalright}

\begin{document}

\title{Central stability homology}

\author{Peter Patzt}
\address{Institut f\"ur Mathematik, Freie Universit\"at Berlin, Germany}
\email{peter.patzt@fu-berlin.de}
\date{April 2017}

\begin{abstract}
We give a new categorical way to construct the central stability homology of Putman and Sam and explain how it can be used in the context of representation stability and homological stability. In contrast to them, we cover categories with infinite automorphism groups. We also connect central stability homology to Randal-Williams and Wahl's work on homological stability. We also develop a criterion that implies that functors that are polynomial in the sense of Randal-Williams and Wahl are centrally stable in the sense of Putman.
\end{abstract}
  
\maketitle
\setcounter{tocdepth}{1}    
\tableofcontents


\section{Introduction}

In 1960 Nakaoka \cite{N} proved the pioneering result that for the symmetric groups
\[ H_i(\mathfrak S_{n-1}) \longrightarrow H_i(\mathfrak S_{n})\]
is an isomorphism for all $n$ large enough in comparison to $i$. Quillen \cite{QuillenSymmetricGroups} was able to prove this result, using a new method that has since been generalized in many ways. He was considering a highly connected simplicial complex $X_n$ on which $\mathfrak S_n$ acts and the groups that stabilizes simplices pointwise are isomorphic to $\mathfrak S_m$ for large enough $m\le n$. For the symmetric groups $X_n = \Delta^{n-1}$ the $(n-1)$--dimensional standard simplex works. Randal-Williams and Wahl \cite{RW} turned this game around: They construct a semisimplicial set $W_n$ whose $p$--simplices are given by 
\[ \mathfrak S_n/\mathfrak S_{n-p-1}\]
such that it is a transitive $\mathfrak S_n$--set whose stabilizers are $\mathfrak S_{n-p-1}$. Then they prove that $W_n$ is highly connected. More precisely, let $\Delta'_+$ be the category of finite ordered sets (including the empty set) and strictly monotone maps and $\FI$ the category of finite sets and injection. Being a semisimplicial set, $W_n$ is a functor from $(\Delta'_+)^{\op}$ to the category of sets with
\[ \{0,1,\dots,p\} \longmapsto \mathfrak S_n/\mathfrak S_{n-p-1}.\]
Because $ \mathfrak S_n/\mathfrak S_{n-p-1}$ is as an $\mathfrak S_n$--set the set of injective maps
\[ \{0,1,\dots,p\} \inject \{ 1, 2, \dots, n\},\]
this functor factors through $\FI^{\op}$ the opposite category of $\FI$:
\[ (\Delta'_+)^{\op} \longrightarrow \FI^{\op} \stackrel{\Hom_{\FI}(-,\{1,\dots,n\})}{\longrightarrow} \Set\]

The category $\FI$ has taken up a prominent role in the theory of representation stability. Established by Church and Farb in \cite{CF}, representation stability characterizes sequences $(V_n)_{n\in \NN}$ of representations of $\mathfrak S_n$ or other sequences of groups, that admit a uniform description as a direct sum of irreducible representation once $n$ is large enough. For example, the permutation representation $\QQ^n$ is the direct sum of the trivial representation and the standard representation for $n\ge2$. The permutation representation can also be realized as an $\FI$--module,
by which we mean a functor from $\FI$ to $\QQ$--vector spaces (and later more generally a functor from $\FI$ to $R$--modules). This can be done by sending a finite set $S$ to the finite dimensional vector space with basis $S$. For an $\FI$--module $V$, we will denote the images of $\{1,\dots, n\}$ by $V_n$.

\subsection*{Notions of representation stability}  Church and Farb's original definition \cite[Def 2.3]{CF} was given in terms of irreducible representations. Later Church, Ellenberg and Farb \cite{CEF} proved that a finitely generated $\FI$--module is representation stable. Putman \cite{P} noticed that for an $\FI$--module $V$ one cannot expect that $V_{n} \cong \Ind^{\mathfrak S_{n}}_{\mathfrak S_{n-1}} V_{n-1}$, because the transposition of $(n-1\ n)$ acts trivially on the image of $V_{n-2}$ in $V_n$. But he calls an $\FI$--module $V$ centrally stable if $V_{n}$ is isomorphic to the largest quotient of $ \Ind^{\mathfrak S_{n}}_{\mathfrak S_{n-1}} V_{n-1}$ such that $(n-1\ n)$ acts trivially on the image of $V_{n-2}$ for all $n$ large enough. This quotient can be expressed as a coequalizer
\[ \coeq\big( \Ind^{\mathfrak S_{n}}_{\mathfrak S_{n-2}} V_{n-2} \rightrightarrows \Ind^{\mathfrak S_{n}}_{\mathfrak S_{n-1}} V_{n-1} \big).\]
Let $\FI^{\le N}$ be the full subcategory of $\FI$ whose objects all have cardinality at most $N$ and denote the inclusion by 
\[ \inc_N\colon \FI^{\le N} \inject \FI.\]
Church, Ellenberg, Farb and Nagpal \cite{CEFN} then proved that an $\FI$--module is centrally stable if and only if it is presented in $\FI^{\le N}$ (having a presentation by sums of representable functors from $\FI^{\le N}$, also see \autoref{prop:presdeg}\eqref{item:presdeg}) for some sufficiently large $N$.
Later Putman and Sam \cite{PS} also gave a categorical concept the name central stability. They call an $\FI$--module $V$ is centrally stable if there is an $N$ large enough such that $V$ is the left Kan extension of its restriction to $\FI^{\le N}$:
\[ V \cong \Lan_{\inc_N} (V\circ \inc_N)\]
In other words, for any $\FI$--module $W$ and any morphism (ie natural transformation)
\[ f \colon V\circ \inc_N \longrightarrow W\circ \inc_N\]
between the restrictions, there is a unique morphism $V\to W$ which induces $f$. In a more general setup, Djament \cite{Dj} proves that an $\FI$--module fulfills this condition if and only if it is presented in $\FI^{\le N}$ and is thereby equivalent to the original definition from \cite{P}. Unfortunately the two notions of central stability differ in natural generalizations of the category $\FI$, we consider later. We compare many different notions in \autoref{sec:notions}.

\subsection*{The central stability chain complex}
The term \emph{central stability chain complex} was first coined by Putman \cite{P} for a sequence $(V_n)_{n\in\NN}$ of $\mathfrak S_n$--representations. He defined it  by
\begin{equation}\label{eq:firstcomplex}
  \dots \to\Ind_{\mathfrak S_{n-p-1}\times \mathfrak S_{p+1}}^{\mathfrak S_n} V_{n-p-1} \otimes \mathcal A_{p+1}\to\Ind_{\mathfrak S_{n-p}\times \mathfrak S_p}^{\mathfrak S_n} V_{n-p} \otimes \mathcal A_p \to \dots
\end{equation}
where $\mathcal A_p$ is the sign representation of $\mathfrak S_p$. It turns out, that for an $\FI$--module $V$ this chain complex is given by
\[  \dots \to \bigoplus_{\substack{S\subset \{1, \dots, n\}\\ |S|=p+1}} V(\{1, \dots, n\} \setminus S)\to \bigoplus_{\substack{S\subset \{1, \dots, n\}\\ |S|=p}} V(\{1, \dots, n\} \setminus S) \to \dots \]
where the differential is the alternated sum of face maps
\[ d_i\colon V\big(\{1, \dots, n\} \setminus \{ j_0< \dots< j_p\} \big) \longrightarrow V\big(\{1, \dots, n\} \setminus \{ j_0< \dots<\hat{j_i}< \dots< j_p\} \big)\]
that is induced by the inclusion of sets. This complex also computes \emph{$\FI$--homology} -- the derived functor analyzed by Church and Ellenberg \cite{CE}.

There is a related chain complex
\begin{equation}\label{eq:secondcomplex}
 \dots \to \!\!\!\!\!\!\!\!\!\!\!\!\!\!\!\! \bigoplus_{f\colon \{0,1,\dots,p\} \inject\{1,\dots, n\}} \!\!\!\!\!\!\!\!\!\!\!\!\!\!\!\!V(\{1,\dots,n\} \setminus \im f) \to  \!\!\!\!\!\!\!\!\!\!\!\!\!\!\!\!\bigoplus_{f\colon \{0,1,\dots,p-1\} \inject\{1,\dots, n\}} \!\!\!\!\!\!\!\!\!\!\!\!\!\!\!\!V(\{1,\dots,n\} \setminus \im f) \to \dots
\end{equation}
for $\FI$--modules $V$. It was first mentioned by Church, Ellenberg, Farb and Nagpal \cite{CEFN} and considered in more detail by Putman and Sam \cite{PS}. One notes
\begin{equation}\label{eq:complexInd}
\bigoplus_{f\colon \{0,1,\dots,p\} \inject\{1,\dots, n\}} V(\{1,\dots,n\} \setminus \im f) \cong \Ind_{\mathfrak S_{n-p}}^{\mathfrak S_n} V(\{1,\dots, n-p\}).
\end{equation}
To recover Putman's chain complex \eqref{eq:firstcomplex}, one factors out the $\mathfrak S_{p+1}$--action.
The second complex \eqref{eq:secondcomplex} computes what we will call \emph{central stability homology} of an $\FI$--module in this paper. 

Both complexes are related to the notions of representation stability. For example, the augmentation map
\[ \Ind^{\mathfrak S_n}_{\mathfrak S_{n-1}} V_{n-1} \longrightarrow V_n\]
is surjective for all $n> N$ if and only if $V$ is generated in $\FI^{\le N}$.

A similar connection can be found for central stability. By definition 
\[ V_n \cong \coeq( \Ind^{\mathfrak S_{n}}_{\mathfrak S_{n-2}} V_{n-2} \rightrightarrows \Ind^{\mathfrak S_{n}}_{\mathfrak S_{n-1}} V_{n-1} )\]
if and only if the part
\[ \Ind_{\mathfrak S_{n-2}}^{\mathfrak S_n} V_{n-2} \to \Ind_{\mathfrak S_{n-1}}^{\mathfrak S_n} V_{n-1} \to V_n \to 0\]
of \eqref{eq:secondcomplex} is exact. Coincidently this is true if and only if the part
\[ \Ind_{\mathfrak S_{n-2}\times \mathfrak S_{2}}^{\mathfrak S_n} V_{n-2} \otimes \mathcal A_{2}\to \Ind_{\mathfrak S_{n-1}\times \mathfrak S_{1}}^{\mathfrak S_n} V_{n-1} \otimes \mathcal A_{1} \to V_n \longrightarrow 0\]
of \eqref{eq:firstcomplex} is exact. The remaining parts of the two complexes may very well differ. 

Such a generalization was done for both complexes by Putman and Sam \cite{PS} for modules over complemented categories (see \autoref{Def:weakly complemented category}). $\FI$ is an example of a complemented category.

Let us plug in the constant $\FI$--module $V(S) = \QQ$ that sends all maps to the identity. We find that \eqref{eq:secondcomplex} is the chain complex associated to $W_n$ from above. Whereas \eqref{eq:firstcomplex} computes to be the simplicial complex of the standard simplex $\Delta^{n-1}$. Randal-Williams and Wahl \cite{RW}  gave a general construction of the semisimplicial set $W_n$ for a homogeneous category (see \autoref{Def:homogeneouscat}). $\FI$ is also an example of a homogeneous category. Moreover, they have a general construction of a simplicial complex $S_n$ whose simplicial chain complex is a generalization of \eqref{eq:firstcomplex}.

Let $\C$ be a small monoidal category with initial unit and let $\D$ be a cocomplete category. In  \autoref{sec:centralstab}, we give a construction of a functor
\[ \Fun(\C,\D)  \longrightarrow \Fun(\C,\Fun((\Delta'_+)^{\op},\D))\]
that turns a functor from $\C$ to $\D$ into a functor from $\C$ to augmented semisimplicial objects in $\D$. This generalizes Randal-Williams and Wahl's construction of the augmented semisimplicial set $W_n$ and the augmented semisimplicial module in \eqref{eq:secondcomplex}. For example plugging in the constant functor
\[ \Hom_{\FI}(\emptyset, -)\colon \FI \longrightarrow \Set,\]
we get a functor from $\FI$ to augmented semisimplicial sets with
\[ \{1,\dots, n\} \longmapsto W_n.\]
In our construction functorial dependencies become more visible than in previously existing literature. If $\D$ is abelian, we denote the associated chain complex for a functor $F\colon \C \to \D$ by $\widetilde C_*(F)$ and its homology by $\widetilde H_*(F)$. This is what we call the \emph{central stability chain complex} and the \emph{central stability homology} in this paper. When $V$ is an $\FI$--modules it computes to be
\[ \widetilde C_p(V)(S) \cong \bigoplus_{f\colon \{0,1,\dots,p\} \inject S} V(S \setminus \im f)\]
as in \eqref{eq:secondcomplex}.

\subsection*{Stability categories}
\cite{PS} and \cite{RW} generalize $\FI$ by weakly complemented categories (see \autoref{Def:weakly complemented category}) and homogeneous categories (see \autoref{Def:homogeneouscat}), respectively. We will work with a different generalization of $\FI$ that we call \emph{stability categories} which are both weakly complemented and homogeneous. In  \autoref{sec:stabcat}, we define \emph{stability groupoids}, that are sequences of groups $\G = (G_n)_{n\in\NN}$ together with a monoidal structure
\[ \oplus \colon G_m\times G_n \inject G_{m+n}\]
and some mild condition (see  \autoref{Def:stability groupoid}). We then use a construction of Quillen, which is also used in \cite{RW}, to get a category $U\G$ whose objects are the natural numbers and whose homomorphisms are
\[ \Hom_{U\G}(m,n) \cong \begin{cases} G_n/G_{n-m} &m\le n\\ \emptyset &m>n.\end{cases}\]
We call $U\G$ a \emph{stability category} if $\G$ is braided. All examples considered by \cite{PS} and  \cite{RW}  are covered in this setup:

\begin{exs}\label{exs:stabcat} Here is a list of some braided stability groupoids $\G$ and their stability categories $U\G$.
\begin{enumerate}
\item The symmetric groups $\mathfrak S= (\mathfrak S_n)_{n\in\NN}$ is a symmetric stability groupoid and its stability category $U\mathfrak S$ is a skeleton of $\FI$.
\item Fix a commutative ring $R$. The general linear groups $\GL(R) = (\GL_n(R))_{n\in\NN}$  is a symmetric stability groupoid and its stability category $U\GL(R)$ is a skeleton of $\VIC_R$ whose objects are finitely generated free $R$--modules and its morphisms $R^m\to R^n$ are given by pairs $(f,C)$ of a monomorphism $f\colon R^m \to R^n$ and a free complement $C\oplus \im f = R^n$.
\item Fix a commutative ring $R$. The symplectic groups $\Sp(R)=(\Sp_{2n}(R))_{n\in \NN}$ is a symmetric stability groupoid and its stability category $U\Sp(R)$ is a skeleton of $\SI_R$ whose objects are finitely generated free symplectic $R$--modules and its morphisms are isometries.
\item The automorphism groups of free groups $\AutF=(\Aut F_n)_{n\in \NN}$ is a symmetric stability groupoid and its stability category $U\AutF$ is given by pairs $(f,C)$ of a monomorphism $f\colon F_m \to F_n$ and a complement $C* \im f = F_n$.
\item There is a braided stability groupoid for the braid groups $\beta = (\beta_n)_{n\in\NN}$. Its monoidal structure is described in \autoref{ex:Symbraid}.
\item There is a braided stability groupoid for the mapping class groups $\Mod\Sigma = (\Mod \Sigma_{g,1})_{g\in \NN}$ of connected, oriented surfaces of genus $g\in\NN$ and one boundary component. Its monoidal structure is described in \cite[Sec 5.6]{RW}. A description of the corresponding stability categories of $\beta$ and $\Mod\Sigma$ can also be found in \cite[Sec 5.6]{RW}.
\end{enumerate}
\end{exs}

From now on, we call a functor from $U\G$ to the category of $R$--modules for some fixed commutative ring $R$ a \emph{$U\G$--module}. In  \autoref{sec:central stability and stab cats}, we compute the central stability complex of a $U\G$--module $V$ very concretely to be isomorphic to 
\[  \widetilde C_p(V)_n \cong \bigoplus_{f\in\Hom_{U\G}(p+1,n)} V_{n-p-1} \cong \Ind_{G_{n-p-1}}^{G_n} V_{n-p-1}.\]
This complex clearly generalizes \eqref{eq:secondcomplex} and \eqref{eq:complexInd}.

The stability category $U\G$ of a braided stability groupoid $\G$ is monoidal (see \cite[Prop 2.6]{RW} or \autoref{prop:stabcat is homogeneous}) and has an initial unit $0$ as
\[ \Hom_{U\G}(0, n) \cong G_n/G_{n-0}\cong \{*\}.\]
That means
\[ R\Hom_{U\G}(0,-) \colon U\G \longrightarrow \xmod{R}\]
is the functor sending all objects to $R$ and all morphisms to the identity on $R$. 
The connectivity condition on $U\G$ used in \cite{RW} can be described by the following condition.
\begin{description}
\item[H3] The central stability homology $\widetilde H_i(R\Hom_{U\G}(0,-))_n =0$ for all $i\ge -1$ and all $n$ large enough in comparison to $i$.
\end{description}

All stability categories in \autoref{exs:stabcat} satisfy \con{H3}. Also, Putman and Sam's \cite[Thm 3.7]{PS} implies \con{H3} for all stability categories $U\G$ that fulfill the noetherian condition that every submodule of a finitely generated $U\G$--module is again finitely generated.
Using \con{H3}  in \autoref{sec:finprop}, we relate resolutions by $U\G$--modules freely generated in finite ranks (see \autoref{Def:Cset}) with the central stability homology. We prove a quantitive version of the following theorem  as  \autoref{thm:res of finite type}, where we also consider partial resolutions.

\begin{thmA}\label{thmA}
Let $U\G$ be a stability category. Assume \con{H3}. Let $V$ be a $U\G$--module,
then the following statements are equivalent.
\begin{enumerate}
\item There is a resolution 
\[ \dots \to P_1 \to P_0 \to V \to 0\]
of $V$ by $U\G$--modules $P_i$ that are freely generated in finite ranks.
\item The homology
\[ \widetilde H_i(V)_n =0\]
for all $i\ge -1$ and all $n$ large enough in comparison to $i$.
\end{enumerate}
\end{thmA}

Let $U\G^{\le N}$ denote the full subcategory of $U\G$ whose objects (which are natural numbers) are exactly $\{0,\dots, N\}$. Previously it has only been known that $\widetilde H_{-1}(V)_n = 0$ for all $n>d$ if and only if $V$ is generated in ranks $\le d$. For $\FI$, our $\widetilde H_{-1}$ is what is denoted by $H_0$ in \cite{CEF}. \cite{CEFN} and \cite{PS} have then used a noetherian property of the categories, they worked with, to show that all central stability homology vanishes for large enough $n$. \autoref{thmA} explains this phenomenon very clearly. It also makes the existence of free resolutions tangible, when such a noetherian condition is not true. For example for $\VIC_\ZZ$ and $\SI_\ZZ$.

We call a $U\G$--module $V$ \emph{centrally stable} if $\widetilde H_{-1}(V)_n = \widetilde H_{0}(V)_n = 0$ for all large enough $n$. This notion and terminology was introduced by Putman \cite{P} for sequences of representations of the symmetric groups. \autoref{thmA} proves that under the assumption of \con{H3}, a $U\G$--module $V$ is centrally stable if and only if $V$ is presented in $U\G^{\le N}$ for some sufficiently large $N$. A $U\G$--module $V$ is presented in a full subcategory with finitely many objects if and only if it is the left Kan extension of its restriction along the inclusion. Unfortunately this notion was also dubbed central stability in \cite{PS} and \cite{CEFN}. See \autoref{sec:notions} for more clarification on different notion used for central stability.

\subsection*{Polynomial functors}

Many interesting examples of $U\G$--modules satisfy polynomiality conditions. For example the dimensions of the permutation representations grow polynomially (even linearly). In fact \cite{CEF} proved that, assuming $R = \QQ$, dimensional growth is eventually polynomial for all finitely generated $\FI$--modules. In  \autoref{sec:poly}, we recall the definition of the \emph{polynomial degree} of a $U\G$--module $V$ from \cite{RW}. We say $V$ has polynomial degree $-\infty$ if it is zero. For $r\ge0$ we say $V$ has polynomial degree $\le r$ if the map
\[ V(0\oplus n) \longrightarrow V( 1\oplus n) \]
is always injective and its cokernel is a $U\G$--module of polynomial degree $\le r-1$. For example the constant $U\G$--module has polynomial degree $\le 0$ and the permutation representations have polynomial degree $\le 1$. 

We formulate a new condition on stability categories, that involves the cokernel 
\[\coker R\Hom_{U\G}(m,-) = \coker\Big(R\Hom_{U\G}(m,-) \to R\Hom_{U\G}(m,1\oplus -)\Big)\]
of the representable functors $R\Hom_{U\G}(m,-)$.
\begin{description}
\item[H4] The central stability homology $\widetilde H_i(\coker R\Hom_{U\G}(m,-))_n =0 $ for all $m\in \NN$, all $i\ge -1$ and all $n$ large enough in comparison to $m$ and $i$.
\end{description}

It is easy to check that 
\[ \coker R\Hom_{U\mathfrak S}(m,-)\]
is free (eg see \cite[Prop 2.2]{CEFN}) and thus \con{H4} follows from \con{H3} for $U\mathfrak S$ (which is equivalent to $\FI$). We prove together with Jeremy Miller and Jennifer Wilson \cite{MPW} that $U\GL(R)$ and $U\Sp(R)$ satisfy \con{H4} if $R$ is a PID. In \autoref{ex:Ubeta} we show that the stability category $U\beta$ of the braid groups does not satisfy \con{H4}. It is unknown whether $U\AutF$ and $U\Mod \Sigma$ satisfy \con{H4}. Assuming \con{H4},
we can make a statement about the central stability homology of $U\G$--module with finite polynomial degree.

\begin{thmA}\label{thmB}
Let $U\G$ be a stability category. Assume \con{H3} and \con{H4}. Let $V$ be a $U\G$--module with finite polynomial degree,
then $\widetilde H_i(V)_n=0$ for all $i \ge -1$ and all $n$ large enough in comparison to $i$.
\end{thmA}

We prove a quantified version of this theorem with \autoref{quantThmB}. This theorem is the main new ingredient of \cite{MPW}, where we together with Jeremy Miller and Jennifer Wilson prove central stability for the second homology groups of Torelli subgroups of automorphism groups of free groups and mapping class groups as well as certain congruence subgroups of general linear groups.

\subsection*{Stability SES}
In  \autoref{sec:stabSES}, we consider families of short exact sequences 
\[ 1 \longrightarrow N_n \longrightarrow G_n \longrightarrow Q_n \longrightarrow 1\]
that are coming from morphisms of stability groupoids
\[ \mathcal N = (N_n)_{n\in \NN} \longrightarrow  \G= (G_n)_{n\in\NN} \longrightarrow \mathcal Q = (Q_n)_{n\in \NN}.\]
We call these families \emph{stability SES}. For every $U\G$--module $V$, there is a $U\mathcal Q$--module $H_i(\mathcal N;V)$ such that
\[ H_i(\mathcal N;V)_n\cong H_i(N_n;V_n).\]

\begin{exs}\label{exs:stabilitySES}
Here is a (very incomplete) list of interesting stability SES's.
\begin{enumerate}
\item The pure braid groups $P\beta = (P\beta_n)_{n\in\NN}$ are the kernels in the sequence
\[ 1 \longrightarrow P\beta \longrightarrow \beta \longrightarrow \mathfrak S \longrightarrow 1.\]
\item\label{item:SESGL} Fix a commutative ring $R$ and an ideal $I\subset R$. The congruence subgroups $\GL(R,I)=(\GL_n(R,I))_{n\in\NN}$ are the kernels in the sequence
\[ 1 \longrightarrow \GL(R,I) \longrightarrow \GL(R) \longrightarrow \overline{\GL(R/I)} \longrightarrow 1,\]
where $\overline{\GL(R/I)}=(\overline{\GL_n(R/I)})_{n\in\NN}$ are the images of $\GL_n(R)$ in $\GL_n(R/I)$.
\item\label{item:SESSp} Fix a commutative ring $R$ and an ideal $I\subset R$. The congruence subgroups $\Sp(R,I)=(\Sp_n(R,I))_{n\in\NN}$ are the kernels in the sequence
\[ 1 \longrightarrow \Sp(R,I) \longrightarrow \Sp(R) \longrightarrow \overline{\Sp(R/I)} \longrightarrow 1,\]
where $\overline{\Sp(R/I)}=(\overline{\Sp_n(R/I)})_{n\in\NN}$ are the images of $\Sp_n(R)$ in $\Sp_n(R/I)$.
\item\label{item:SESIA} The Torelli subgroups $\IA= (\IA_n)_{n\in\NN}$ of the automorphism groups $\Aut F_n$ are the kernels in the sequence
\[ 1 \longrightarrow \IA \longrightarrow \AutF \longrightarrow \GL(\ZZ) \longrightarrow 1.\]
\item\label{item:SESI} The Torelli subgroups $\mathcal I= (\mathcal I_{g,1})_{g\in\NN}$ of the mapping class groups $\Mod \Sigma_{g,1}$ are the kernels in the sequence
\[ 1 \longrightarrow \mathcal I \longrightarrow \Mod\Sigma \longrightarrow \Sp(\ZZ) \longrightarrow 1.\]
\item If there is an inclusion of $\mathfrak S\subset \mathcal Q$, we can twist $\mathcal N$ to receive a braided stability groupoid $\widetilde{\mathcal N} = (\widetilde N_n)_{n\in\NN}$, where $\widetilde N_n\subset G_n$ is the preimage of $\mathfrak S_n$. Then
\[ 1 \longrightarrow \mathcal N \longrightarrow \widetilde{\mathcal N} \longrightarrow \mathfrak S \longrightarrow 1\]
is a stability SES. This has been done by Putman \cite{P} for \eqref{item:SESGL} and by the author together with Wu \cite{PW} for the Houghton groups. It could also be done for \eqref{item:SESSp}, \eqref{item:SESIA}, and \eqref{item:SESI}.
\end{enumerate}
\end{exs}

Note that in all these examples $\G$ and $\mathcal Q$ are braided, but $\mathcal N$ is not. 
Assume for the following that $\G$ and $\mathcal Q$ both are braided. The central technical tools of this paper are the following spectral sequences. Let $V$ be a $U\G$--module and $n\in\NN$, then there are two spectral sequences
\[ E^1_{pq} = E_pG_n \otimes_{RN_n} \widetilde H_q(V)_n\]
and
\[ E^2_{pq} = \widetilde H_p(H_q(\mathcal N;V))_n\]
that converge to the same limit. In particular when $\widetilde H_q(V)_n=0$ for all  $n$ large enough in comparison to $q$, then both spectral sequences converge to zero in each diagonal $p+q$ for large enough $n$.

As an application of this spectral sequence, we prove a generalization of a theorem of Putman and Sam \cite[Thm 5.13]{PS}.

\begin{thmA}\label{thmC}
Let
\[ 1 \longrightarrow \mathcal N \longrightarrow \mathcal G \longrightarrow \mathcal Q \longrightarrow 1 \]
be a stability SES. Assume that  $\G$, $\mathcal Q$, and $\G\to\mathcal Q$ are braided. Let $V$ be a $U\G$--module. Assume furthermore:
\begin{enumerate}
\item All submodules of finitely generated $U{\mathcal Q}$--modules are finitely generated. (Noetherian condition)
\item $H_i(\mathcal N;V)_n$ is a finitely generated $R$--module for all $i,n\in \NN$.
\item $V$ is a finitely generated $U\G$--module.
\item $\widetilde H_i(V)_n = 0$ for all $n$ large enough in comparison to $i$.
\end{enumerate}
Then $H_i(\mathcal N;V)$ is a finitely generated $U{\mathcal Q}$--module for every $i\in \NN$.
\end{thmA}

This spectral sequence is also an important tool in \cite{MPW}, where we apply it to the stability SES's in \autoref{exs:stabilitySES} \eqref{item:SESGL}, \eqref{item:SESIA}, and \eqref{item:SESI}. 

Finally in  \autoref{sec:Quillen}, we revisit Quillen's argument for homological stability. We prove a theorem on homological stability with twisted coefficients. Randal-Williams and Wahl \cite[Thm A]{RW} prove a similar result but require their coefficients to have finite polynomial degree. If we additionally assume \con{H4}, our \autoref{thmB} implies that our condition is weaker, but in general our ranges are worse than theirs.

\begin{thmA}\label{thmD}
Let $U\G$ be a stability category. Assume \con{H3}. Let $V$ be a $U\G$--module with
\[ \widetilde H_i(V)_n = 0 \]
for all $n\ge k\cdot i + a$ for some $a\in\ZZ$ and $k\ge2$, then the stabilization map
\[ \phi_* \colon H_i(G_n; V_n) \longrightarrow H_i(G_{n+1};V_{n+1})\]
is an epimorphism for all $n\ge k\cdot i +a-1$ and an isomorphism for all $n\ge k\cdot i +a$.
\end{thmA}

\begin{rem}
Notice that \autoref{thmD} and \autoref{rem:shiftfree} give criterions to check whether the double cosets
\[ G_{n-a}\backslash G_n / G_{n-b} \] 
stabilize with growing $n$ and fixed $a,b$. 
\end{rem}

\subsection*{Acknowledgement}
The author was supported by the Berlin Mathematical School and the Dahlem Research School. The author also wants to thank Aur$\acute{\text{e}}$lien Djament, Daniela Egas Santander, Reiner Hermann, Henning Krause, Daniel L\"utgehet\-mann, Jeremy Miller, Holger Reich, Steven Sam, Elmar Vogt, Nathalie Wahl, Jenny Wilson for helpful conversations. Special thanks to Reiner Hermann and his invitation to NTNU where the idea for this project was born.


\section{Central stability homology}\label{sec:centralstab}

Let $(\C,\oplus,0)$ be a small monoidal category with unit object $0$. In particular,
\[ \oplus \colon \C \times \C \to \C\]
is a bifunctor. Let $\D$ be cocomplete category and $\Fun(\C,\D)$ the category of functors from $\C$ to $\D$. Then the monoidal structure gives rise to a functor
\[ \Fun(\C,\D) \to \Fun(\C\times \C, \D)\]
by precomposition of $\oplus$. This is the same as giving a functor
\[ S\colon \C \to \End(\Fun(\C,\D))\]
from $\C$ to the category of endofunctors of $\Fun(\C,\D)$. Explicitly it sends an object $X\in \C$ to the \emph{suspension} endofunctor
\[ S_X \colon F \mapsto F(- \oplus X).\]

Because $S_X$ is just the precomposition of $-\oplus X$, the functor $S_X$ has a left adjoint, namely the left Kan extension along $-\oplus X$ (see \cite[Cor X.3.2]{ML}). We will call its left adjoint the \emph{desuspension} endofunctor and denote it by $\Sigma_X$. 
There is an antiequivalence of categories between left and right adjoint endofunctors. (Cf \cite[Thm 3.2]{KanAF}) Consequently, there is a functor
\[ \Sigma\colon \C^{\op} \to \End(\Fun(\C,\D))\]
that sends $X$ to $\Sigma_X$.

Clearly $S_{X\oplus Y} = S_{Y}\circ S_X$. Because $\Sigma_X$ is the right adjoint of $S_X$
\[ \Sigma_{X\oplus Y} \cong \Sigma_X \circ \Sigma_Y.\]
Thus $\Sigma$ is in fact a monoidal functor, where $\End(\Fun(\C,\D))$ is equipped with the monoidal structure $\circ$.

\begin{ex}
Consider the monoidal category $\C = \FI$ (which is not small but has a small skeleton) and $\D = \xmod{\ZZ}$ the cocomplete category of abelian groups. Then $\Sigma_X$ is given by
\[ (\Sigma_X F)(A) \cong \bigoplus_{f\colon X \inject A} F(A \setminus \im f),\]
for finite sets $X$ and $A$ and an $\FI$--module $F\colon \FI \to \xmod{\ZZ}$. This $\ZZ$--module is isomorphic to 
\[ \Ind_{\mathfrak S_{n-m}}^{\mathfrak S_n} F(\{1, \dots, n-m\})\]
when $X = \{1,\dots,m\}$ and $A = \{ 1, \dots, n\}$ with $m \le n$.
\end{ex}

\begin{Def}
Let the \emph{augmented semisimplicial category $\Delta'_+$} be the category with objects the sets $\{0,\dots, n\}$ for $n\ge -1$ and morphisms strictly increasing set maps. \[\{ 0, \dots, m\} \oplus \{0, \dots, n\} = \{0, \dots, m, m+1, \dots, m+1+n\}\] defines a monoidal structure on $\Delta'_+$ with the unit $\emptyset$, which is additionally initial.
\end{Def}

The following theorem describes how $\Delta'_+$ behaves universally among all monoidal categories with initial unit object. Let for the remainder of this section $(\C,\oplus, 0)$ be a small monoidal category with initial unit object $0$ and denote the initial maps by $\iota_X\colon 0\to X$. 

\begin{thm}
Let $X$ be an object in $\C$. Then there exists a unique monoidal functor $\Delta'_+ \to \C$ that sends $\{0\}$ to $X$.
\end{thm}

\begin{proof}
Similar to \cite[Ex 8.1.6]{Wei}, a functor $\Delta'_+ \to \C$ is given by a sequence of objects $K_{-1},K_0, \dots$ and maps $d_i\colon K_{p-1} \to K_p (i = 0,\dots,p)$ such that if $i<j$ then $d_jd_i = d_id_{j-1}$. To be monoidal $K_p$ must be $X^{\oplus p+1}$ for all $n \ge -1$. Because $0\in \C$ is initial, there is exactly one candidate for $d_0\colon K_{-1} =0 \to K_0 = X$, which is the initial map $\iota_X \colon 0 \to X$. Further the other maps must have the following form.
\[ d_i = \id_{X^{\oplus i}}\oplus \iota_X \oplus \id_{X^{\oplus p-i}} \colon K_{p-1} = X^{\oplus i} \oplus 0 \oplus X^{\oplus p-i} \to K_p = X^{\oplus i} \oplus X \oplus X^{\oplus p-i}\]
One easily checks that this satisfies the condition given.
\end{proof}

\begin{Def}\label{Def:semisimplicialobject}
Let $X$ an object of $\C$. Then we get an augmented semisimplicial object
\[{\Delta'_+}^{\op} \to  \C^{\op} \to \End(\Fun(\C,\D)).\]
This is the same as a functor
\[ K^X\colon \Fun(\C,\D) \longrightarrow \Fun({\Delta'_+}^{\op},\Fun(\C,\D)).\]
Thus for every functor $F\colon \C\to \D$ we get an augmented  semisimplicial object $K_\bullet^XF$ in $\Fun(\C,\D)$ with
\[ K^X_pF = \Sigma_{X^{\oplus p+1}}F \cong \Sigma^{p+1}_X F.\]
\end{Def}

\begin{Def}\label{Def:centralstabilityhomology}
Assume $\D$ is an abelian category, we denote the associated chain complex of $K_\bullet^XF$ by $\widetilde C^X_*F$ and the homology of this chain complex by $\widetilde H^X_* F$. We call this the \emph{central stability homology} of $F$.
\end{Def}

%

A well known consequence of the Yoneda Lemma is that
\[ \Sigma_X \Hom_\C(A,-) \cong \Hom_\C(A\oplus X,-).\]

\begin{ex}
Let us consider $\C=\FI$ and $\mathcal D = \Set$, and let us fix $X=\{*\}$ a singleton. Given $F=\Hom_{\FI}(\emptyset,-)$ the representable functor, $K_\bullet^X F$ is given by
\[ K_p^X F \cong \Hom_{\FI}(\{0,\dots,p\},-)\]
with face maps given by precomposition of the morphisms
\[ \{0,\dots,\hat i, \dots, p \} \inject \{ 0, \dots, p\}\]
where the hat indicates that $i$ is omitted.

Evaluating at $\{1,\dots, n\}$, we get the (augmented) semisimplicial set
\[ W_n \cong (K_\bullet^X F)(\{1,\dots,n\})\]
from \cite[Def 2.1]{RW} associated to the homogeneous category $\FI$.
\end{ex}

More generally we can describe $\Sigma_X F$ for a functor $F\colon \C\to \D$ as a colimit. Let $(-\oplus X\downarrow B)$ be the comma category whose objects are the pairs $(A, \psi \colon A\oplus X \to B)$, then there is a forgetful functor
\[ (A, \psi\colon A\oplus X \to B) \mapsto A.\]
Precomposing $F$ with this functor we get
\[ (\Sigma_X F)(B) \cong \colim_{(-\oplus X \downarrow B)} F.\]

For a morphism $X \to Y$ in $\C$, the natural transformation $\Sigma_YF \to \Sigma_X F$ is given by the functor
\[ (-\oplus Y\downarrow B) \longrightarrow (-\oplus X \downarrow B)\]
mapping to the composition
\[ (A\oplus Y \to B) \longmapsto (A\oplus X \to A\oplus Y \to B).\]

\begin{ex}\label{ex:FIhomology}
For $\C=\FI$, $\mathcal D = \xmod\ZZ$ and $X=\{*\}$, the central stability homology  is given by the semisimplicial $\FI$--module $K_\bullet^X F$ with
\[ K_p^X F \cong \bigoplus_{f\in \Hom_{\FI}( \{0,\dots, p\},-)} F(C_f)\]
where $C_f$ is the complement of the image of $f$ and the face maps given by the precomposition of the morphisms above and the inclusion of the complements.
%
\end{ex}


\section{Stability categories}\label{sec:stabcat}

In this section we introduce stability categories that are a specific kind of homogeneous categories (see \cite[Def 1.3]{RW} or \autoref{Def:homogeneouscat}) and weakly complemented categories (see \cite{PS} or \autoref{Def:weakly complemented category}).

\begin{Def}\label{Def:stability groupoid}
Let $(\G,\oplus,0)$ be a monoidal groupoid whose monoid of objects is the natural numbers $\NN$. The automorphism group of the object $n \in \NN$ is denoted $G_n = \Aut^\G(n)$. Then $\G$ is called a \emph{stability groupoid} if it satisfies the following properties.
\begin{enumerate}
\item The monoidal structure
\[ \oplus \colon G_m\times G_n \inject G_{m+n}\]
is injective for all $m,n\in \NN$.
\item The group $G_0$ is trivial.
\item $(G_{l+m}\times 1)\cap (1\times G_{m+n}) = 1\times G_{m} \times 1\subset G_{l+m+n}$ for all $l,m,n\in \NN$.
\end{enumerate}
\end{Def}

\begin{Def}
A \emph{homomorphism} $ \mathcal G \to \mathcal H$ of stability groupoids is a monoidal functor sending $1$ to $1$.
\end{Def}

\begin{rem}
A stability groupoid $\G$ is given by a sequence  of groups $(G_n)_{n \in \NN}$ and group homomorphisms
\[ \oplus\colon G_m \times G_n \longrightarrow G_{m+n}\]
satisfying additional properties that would require spelling out what a monoidal structure is. In notation we will usually suppress the monoidal structure and just write $\G = (G_n)_{n\in \NN}$ for a stability groupoid.

Similarly a homomorphism of stability groupoids  
\[f\colon \G=(G_n)_{n\in\NN} \longrightarrow \mathcal H= (H_n)_{n\in\NN}\]
 is given by a sequence of group homomorphisms
\[ f_n\colon G_n \longrightarrow H_n\]
for which the diagrams
\[ \xymatrix{
G_m\times G_n \ar[r]\ar[d]_{f_m\times f_n} & G_{m+n}  \ar[d]^{f_{m+n}}\\
H_m\times H_n \ar[r] & H_{m+n}
}\]
commute.
\end{rem}

Recall that a \emph{braiding} $b$ of a monoidal category $(\C,\oplus, 0)$ is a natural isomorphism between the functors $\oplus \colon \C\times \C \to \C$ and $\oplus$ precomposed with a swap that has some additional conditions. In particular, there are isomorphisms 
\[b_{A,B}\colon A\oplus B \to B\oplus A\]
 for all objects $A,B\in \C$, such that for all morphisms $f\colon A\to C$ and $g\colon B \to D$ the diagram
\[ \xymatrix{
A\oplus B \ar[r]^{b_{A,B}} \ar[d]_{f\oplus g} & B\oplus A \ar[d]^{g \oplus f}\\
C\oplus D \ar[r]_{b_{C,D}} & D\oplus C
}\]
commutes. For such natural isomorphism to be a braiding one also needs two hexagon identities. (Cf \cite[Sec XI.1]{ML} for details.) If
\[ b^{-1}_{A,B} = b_{B,A}\]
the braiding is called a \emph{symmetry}.

For a stability groupoid $\G = (G_n)_{n\in\NN}$, a braiding $b$ is given by isomorphisms
\[ b_{m,n} \in G_{m+n}\]
such that 
\[ (g \oplus f)\circ b_{m,n} = b_{m,n} \circ (f\oplus g)\]
for all $f\in G_m$ and $g\in G_n$ and
\[b_{l,m+n} = (\id_m\oplus b_{l,n})\circ(b_{l,m} \oplus \id_n)\quad\text{and}\quad b_{l+ m,n } = (b_{l,n}\oplus \id_m) \circ (\id_l\oplus b_{m,n}) \]
which amounts to the hexagon identities.

For any monoidal groupoid $\G$, Randal-Williams and Wahl \cite[Sec 1.1]{RW} give a construction originally due to Quillen that yields the category $U\G$. 

\begin{Def}
Let $(\G,\oplus,0)$ be a monoidal groupoid. The category $U\G$ has the same objects as $\G$ and its morphisms from $A$ to $B$ are equivalence classes of pairs $(f,C)$ where $C$ is an object in $\G$ and $f$ is an (iso)morphism in $\G$ from $C\oplus A \to B$. Two of these  pairs $(f,C)$ and $(f',C')$ are equivalent if there is an isomorphism $g\colon C \to C'$ (in $\G$) such that the diagram
\[ \xymatrix{
C\oplus A \ar[r]^f \ar[d]_{g\oplus \id_A} & B\\
C'\oplus A \ar[ru]_{f'}
}\]
commutes.
\end{Def}

\begin{Def}
If $\G$ is a braided stability groupoid, we call $(U\G,\oplus,0)$ the \emph{stability category} of $\G$.
\end{Def}

The following definition is due to Randal-Williams and Wahl \cite[Def 1.3]{RW}.

\begin{Def}\label{Def:homogeneouscat}
Let  $(\C,\oplus,0)$ be a monoidal category with initial $0$. $\C$ is called a \emph{homogeneous category} if it satisfies the following conditions for all objects $A,B\in \C$:
\begin{description}
\item[H1] $\Aut(B)$ acts transitively on $\Hom(A,B)$ by postcomposition.
\item[H2]  The map $\Aut(A) \to \Aut(A\oplus B)$ taking $f$ to $f\oplus \id_B$ is injective and its image is
\[ \Fix(B,A\oplus B) := \{ f\in \Aut(A\oplus B) \mid f\circ (\iota_A \oplus \id_B) = \iota_A \oplus \id_B\},\]
where $\iota_A\colon 0 \to A$ denotes the initial morphism.
\end{description}

$\C$ is called \emph{prebraided} if its underlying groupoid is braided with a braiding $b$, that satisfies the equation
\[ b_{A,B} \circ (\id_A \oplus \iota_B) = \iota_B \oplus \id_A\]
for all objects $A,B\in \C$.
\end{Def}

\begin{rem}\label{rem:symH2}
The condition \con{H2} is not only symmetric for symmetric homogeneous category $(\C,\oplus,0)$; prebraided suffices for this conclusion. Here is the reason. Clearly $\Aut(A)\to \Aut(B\oplus A)$ is also injective and the image lies in
\[ \Fix(B, B\oplus A) = \{ f\in \Aut(B\oplus A) \mid f\circ (\id_B\oplus \iota_A) = \id_B \oplus \iota_A\}.\]
But assume $f\in \Fix(B,B\oplus A)$ then $b_{B,A}\circ f \circ b^{-1}_{B,A} = g \oplus B$ for some $g\in \Aut(A)$ by \con{H2} and thus
\[ f = b_{B,A}^{-1} \circ (g\oplus \id_B) \circ b_{B,A} = \id_B \oplus g.\]
\end{rem}

The stability category of a braided stability groupoid is always a homogeneous category.

\begin{prop}\label{prop:stabcat is homogeneous}
Let $\G$ be a braided (symmetric) stability groupoid, then $U\G$ is a prebraided (symmetric) homogeneous category. And the underlying groupoid of $U\G$ is $\G$. 
\end{prop}

\begin{proof}
In \cite[Prop 2.6(i)+(ii)]{RW} it is shown that $0$ is initial in $U\G$ and that $U\G$ is prebraided. For \con{H1} and \con{H2} the statements \cite[Thm 1.8(c)+(d)]{RW} apply. For the second statement \cite[Prop 2.10]{RW} applies.
\end{proof}

In \cite[Rem 1.4]{RW} it is already stated that in every homogeneous category
\[ \Hom(B,A\oplus B) \cong \Aut(A\oplus B)/ \Aut(A).\]
For $U\G$ this means
\[ \Hom(n,m\oplus n) \cong G_{m+ n}/G_m.\]
Because there is cancellation in $\NN$, the homomorphisms are given by equivalence classes of pairs $(f,m)$ for some $f\in G_{m+n}$ and the above isomorphism is given by
\[ [f,m] \longmapsto fG_m.\]
The composition is then
\[ fG_l \circ gG_{m} = f(\id_l \oplus g)G_{l+m}\]
for $fG_{l}\colon m\oplus n \to l \oplus m \oplus n$ and $gG_m\colon n \to m \oplus n$.

The monoidal structure on $U\G$ is then given by
\[ f_1G_{m_1}\oplus f_2G_{m_2} = (f_1\oplus f_2)(\id_{m_1} \oplus b^{-1}_{n_1,m_2} \oplus \id_{n_2})G_{m_1+m_2}\]
for $f_1G_{m_1}\in \Hom(n_1,m_1\oplus n_1)$ and $f_2G_{m_2} \in \Hom(n_2,m_2\oplus n_2)$.

\begin{ex}\label{ex:Symbraid}
$\mathfrak S = (\mathfrak S_n)_{n\in\NN}$ is a braided stability groupoid with the braiding $b_{m,n}\in \mathfrak S_{m+n}$ given by the permutation
\[ b_{m,n}(i) =\begin{cases} i+n & i\le m\\ i-m &i>m.\end{cases}\]
In fact, the braiding of $\mathfrak S$ is a symmetry, wherefore it is also a symmetry on $U\mathfrak S$.

An example of a braided stability groupoid that is not symmetric is given by the braid groups $\beta = (\beta_n)_{n\in \NN}$. Its braiding $b_{m,n} \in \beta_{m+n}$ is given by the following diagram.
\[ b_{m,n}  = \quad 
\begin{tikzpicture}[x=.5cm,y=-1.5cm,baseline=-.8cm]

\node[v] (a1) at (0,0) {};
\node at (1,0) {$\cdots$};
\node[v] (a2) at (2,0) {};
\node[v] (a3) at (3,0) {};
\node at (4,0) {$\cdots$};
\node[v] (a4) at (5,0) {};

\draw [decorate,decoration={brace,amplitude=4pt,raise=4pt},yshift=0pt]
(0,0) -- (2,0) node [black,midway,yshift=12pt] {\scalebox{.8}{$m$}};
\draw [decorate,decoration={brace,amplitude=4pt,raise=4pt},yshift=0pt]
(3,0) -- (5,0) node [black,midway,yshift=12pt] {\scalebox{.8}{$n$}};
\draw [decorate,decoration={brace,amplitude=4pt,raise=4pt},yshift=0pt]
(2,1) -- (0,1) node [black,midway,yshift=-12pt] {\scalebox{.8}{$n$}};
\draw [decorate,decoration={brace,amplitude=4pt,raise=4pt},yshift=0pt]
(5,1) -- (3,1) node [black,midway,yshift=-12pt] {\scalebox{.8}{$m$}};

\node[v] (b1) at (0,1) {};
\node at (1,1) {$\cdots$};
\node[v] (b2) at (2,1) {};
\node[v] (b3) at (3,1) {};
\node at (4,1) {$\cdots$};
\node[v] (b4) at (5,1) {};

\draw[very thick] (a1) to[out=270, in=90] (b3);
\draw[very thick] (a2) to[out=270, in=90] (b4);
\draw[white , double=black , very thick , double distance =1.3pt] (a3) to[out=270, in=90] (b1);
\draw[white , double=black , very thick , double distance =1.3pt] (a4) to[out=270, in=90] (b2);
\end{tikzpicture}
\quad \neq \quad
\begin{tikzpicture}[x=.5cm,y=-1.5cm,baseline=-.8cm]

\node[v] (a1) at (0,0) {};
\node at (1,0) {$\cdots$};
\node[v] (a2) at (2,0) {};
\node[v] (a3) at (3,0) {};
\node at (4,0) {$\cdots$};
\node[v] (a4) at (5,0) {};

\draw [decorate,decoration={brace,amplitude=4pt,raise=4pt},yshift=0pt]
(0,0) -- (2,0) node [black,midway,yshift=12pt] {\scalebox{.8}{$m$}};
\draw [decorate,decoration={brace,amplitude=4pt,raise=4pt},yshift=0pt]
(3,0) -- (5,0) node [black,midway,yshift=12pt] {\scalebox{.8}{$n$}};
\draw [decorate,decoration={brace,amplitude=4pt,raise=4pt},yshift=0pt]
(2,1) -- (0,1) node [black,midway,yshift=-12pt] {\scalebox{.8}{$n$}};
\draw [decorate,decoration={brace,amplitude=4pt,raise=4pt},yshift=0pt]
(5,1) -- (3,1) node [black,midway,yshift=-12pt] {\scalebox{.8}{$m$}};

\node[v] (b1) at (0,1) {};
\node at (1,1) {$\cdots$};
\node[v] (b2) at (2,1) {};
\node[v] (b3) at (3,1) {};
\node at (4,1) {$\cdots$};
\node[v] (b4) at (5,1) {};

\draw[very thick] (a3) to[out=270, in=90] (b1);
\draw[very thick] (a4) to[out=270, in=90] (b2);
\draw[white , double=black , very thick , double distance =1.3pt] (a1) to[out=270, in=90] (b3);
\draw[white , double=black , very thick , double distance =1.3pt] (a2) to[out=270, in=90] (b4);
\end{tikzpicture}
\quad=b_{n,m}^{-1}
\]
\end{ex}

A stability category of a braided stability groupoid is also a weakly complemented category as defined by Putman and Sam \cite{PS}.

\begin{Def}\label{Def:weakly complemented category}
Let  $(\C,\oplus,0)$ be a monoidal category with initial $0$. $\C$ is called a \emph{weakly complemented category} if it satisfies the following conditions for all objects $A,B,C\in \C$:
\begin{enumerate}
\item All morphisms are monomorphisms.
\item \label{item:complinj} The map $\Hom(A\oplus B, C) \to \Hom(A,C) \times \Hom(B,C)$ given by
\[ \psi \mapsto (\psi \circ (\id_A \oplus \iota_B), \psi \circ (\iota_A \oplus \id_B))\]
is injective.
\item For every morphism $\psi\in \Hom(A,B)$ there is an object $D$ and an isomorphism $f\colon D\oplus A \to B$ such that $\psi = f\circ (\iota_D \oplus \id_A)$. The pair $(D,f)$ is uniquely determined up to isomorphisms in the comma category $(-\oplus A\downarrow B)$.
\end{enumerate}
\end{Def}

\begin{prop}
Let $\G=(G_n)_{n\in\NN}$ be a braided stability groupoid, then $U\G$ is a weakly complemented category.
\end{prop}

\begin{proof}
Let $l,m,n\in \NN$, ie objects of $\G$ and $U\G$. To prove (a), by \con{H1}, we only have to show that $\phi = \id_{l\oplus m \oplus n}G_l$ is a monomorphism. Let $\psi = fG_m, \psi'=f'G_m\colon n \to m \oplus n$ such that
\[  (\id_l \oplus f)G_{l+ m} =\phi \circ  \psi   = \phi\circ \psi'=   (\id_l \oplus f')G_{l+ m}.\]
Then there is an automorphism $g\in G_{l+m}$ such that
\[ \id_l\oplus f = (\id_l \oplus f') \circ (g\oplus \id_n).\]
By definition 
\[ g\oplus \id_n = \id_l \oplus ({f'}^{-1}\circ f) \in (G_{l+m} \times 1) \cap (1\times G_{m+n}) =  1 \times G_m \times 1.\]
Therefore $\psi$ and $\psi'$ must coincide.

Using \con{H1} to prove (b), it is enough to show that for every $g\in G_{l+m+n}$ and $\phi=  \id_{l\oplus m \oplus n}G_l$ the two equations
\[ \phi\circ (\id_m \oplus \iota_n) = g\phi \circ (\id_m \oplus \iota_n) \quad\text{and}\quad  \phi \circ(\iota_m \oplus \id_n) = g\phi\circ(\iota_m \oplus \id_n)\]
imply $\phi= g\phi$. Using the second equation we get that
\[ \iota_{l\oplus m} \oplus \id_n = g(\iota_{l\oplus m} \oplus \id_n)\]
thus by \con{H2} there is a $g' \in G_{l+m}$ with $g = g' \oplus \id_n$. Applying the same trick to the first equation we get
\[ \iota_l \oplus \id_m = g (\iota_l \oplus \id_m)\]
and therefore the existence of a $g'' \in G_l$ with $g' = g'' \oplus \id_m$. For $g = g'' \oplus \id_m \oplus \id_n$ clearly $\phi =g\phi$.

(c) is true by definition of $U\G$.
\end{proof}

\noindent\fbox{\parbox[b]{\textwidth}{\textit{For the remainder of this paper we assume that $\G=(G_n)_{n\in\NN}$ is a braided stability  groupoid and thus $U\G$ is its prebraided stability category.  Given a functor $F$ from $U\G$ to some category $\D$, we denote the images $F(n)$ by $F_n$.}

\textit{Further for the object $1\in U\G$, we abbreviate $S_1,\Sigma_1,K^1_\bullet,\widetilde C_*^1, \widetilde H_*^1$ by $S,\Sigma,K_\bullet,\widetilde C_*,\widetilde H_*$, respectively.}}}

\section{Central stability and stability categories}\label{sec:central stability and stab cats}

In this section, we explain some basic properties of the functor $K_\bullet$ from \autoref{Def:centralstabilityhomology} in the context of stability categories.

\begin{prop}\label{prop:Hominj}
Let $l,m,n$ be objects in $U\G$. Then the map $\Hom(m,l\oplus m) \to \Hom(m\oplus n,l\oplus m\oplus n)$ given by $\psi \mapsto \psi \oplus \id_n$ is injective and its image is
\[ \{\chi \in \Hom(m\oplus n,l\oplus m\oplus n) \mid  \chi\circ (\iota_m \oplus \id_n) = \iota_{l\oplus m} \oplus \id_n\}.\]
(Because of  \autoref{rem:symH2}, the symmetric version of this statement is also true.)
\end{prop}

\begin{proof}
Given two morphisms from $fG_l, f'G_l \colon m \to l\oplus m$, and assume that 
\[(f \oplus \id_n)G_l = fG_l\oplus \id_n =  f'G_l\oplus \id_n = (f'\oplus \id_n)G_l.\] 
Then there is an automorphism $g\in G_l$ such that 
\[f\oplus \id_n = (f'\oplus \id_n) \circ (g \oplus \id_{m\oplus n})\]
which implies that
\[ g \oplus \id_{m \oplus n} = ({f'}^{-1}\circ f) \oplus \id_n.\]
Because $G_{l+m} \to G_{l+m+n}$ is injective,
\[ g \oplus \id_m  = {f'}^{-1}\circ f.\]
Thus 
\[  fG_l = f'G_l\]
and the map 
\[ \psi \longmapsto \psi \oplus \id_n\]
is injective.

If $\psi\in \Hom(m,m\oplus n)$,
\[ (\psi \oplus \id_n)\circ (\iota_m \oplus \id_n)  = (\psi \circ \iota_m) \oplus \id_n = \iota_{l\oplus m} \oplus \id_n.\]


Let on the other hand $\chi = hG_l \colon  m \oplus n \to l\oplus m\oplus n$ such that
\[ \chi \circ (\iota_m \oplus \id_n) = \iota_{l\oplus m} \oplus \id_n.\]
Then 
\[ hG_{l+m} = \id_{l\oplus m\oplus n}G_{l+m},\]
whence there is an isomorphism $g\in G_{l+m}$ with $h = g \oplus \id_n$. Thus 
\[ gG_l \oplus \id_n = hG_{l}= \chi.\qedhere\]
\end{proof}

\begin{prop}\label{prop:sigmap}
Let $F \colon U\G \to \Set$. Then there is an isomorphism
\[  (K_{p-1}F)_n = (\Sigma^p  F)_n \cong G_n\times_{G_{n-p}}F_{n-p}.\]

For a map $\eta = hG_{p-q}\in \Hom(q,p)$ the corresponding morphism 
\[  (\Sigma^p  F)_n \cong G_n\times_{G_{n-p}}F_{n-p} \longrightarrow G_n\times_{G_{n-q}}F_{n-q}\cong(\Sigma^qF)_n\]
is given by 
\[[g,x] \longmapsto [g(\id_{n-p}\oplus h),\phi'(x)]\]
with $\phi' = F(\id_{n-p}\oplus \iota_{p-q})$. This morphism is independent of the choice of $h\in \eta$.
\end{prop}

\begin{proof}
We know from  \autoref{sec:centralstab} that
\[ (\Sigma^p  F)_n \cong  \colim_{(-\oplus p\downarrow n)} F.\]
Note that by construction, every morphism $m\oplus p \to n$ factors through an automorphism $(n-p)\oplus p \to n$. Therefore there is a surjection
\[G_n \times F_{n-p} \surject \colim_{(-\oplus p\downarrow n)} F.\]
There are still some relations, which are given by the precomposition by elements in $G_{n-p}$. Therefore
\[  (\Sigma^p  F)_n \cong G_n\times_{G_{n-p}}F_{n-p}\]
together with the (choice-free) maps
\[ F(m,\psi)=F_m \longrightarrow G_n\times_{G_{n-p}}F_{n-p}\]
given by
\[ x \mapsto [g,\phi(x)]\]
where $\phi = F(\iota_{n-p-m} \oplus \id_{ m})$ and $\psi= gG_{n-p-m}$.

To understand the functoriality, we see that 
\[  (-\oplus p\downarrow n) \longrightarrow (-\oplus q\downarrow n)\]
is given by
\[ (m,\psi) \longmapsto (m,\psi\circ (\id_m\oplus\eta) ).\]
Therefore 
\[ [g,x] \in G_n\times_{G_{n-p}}F_{n-p}\]
maps to 
\[ \colim_{(-\oplus q\downarrow n)} F\]
via $F_{n-p}$ corresponding to $(n-p,g\circ (\id_{n-p}\oplus \eta)) \in (-\oplus q \downarrow n)$. Thus 
\[ [g,x] \longmapsto [g(\id_{n-p}\oplus h)\circ(b^{-1}_{n-p,p-q}\oplus \id_q), \phi(x)] = [g(\id_{n-p}\oplus h), \phi'(x)]\] 
where $h\in G_p$ with $\eta = hG_{p-q}$.
\end{proof}

A similar proposition can be made for functors $U\G\to \xmod{R}$:

\begin{prop}\label{prop:SigmaMod}
Let $V\colon U\G \to \xmod{R}$. Then there is an isomorphism
\[ (\Sigma^pV)_n \cong RG_n \otimes_{RG_{n-p}} V_{n-p}.\]

For a map $\eta = hG_{p-q}\in \Hom(q,p)$ the corresponding morphism 
\[  (\Sigma^p  V)_n \cong RG_n\otimes_{RG_{n-p}}V_{n-p} \longrightarrow RG_n\otimes_{RG_{n-q}}V_{n-q}\cong(\Sigma^qV)_n\]
is given by 
\[g\otimes v \longmapsto g(\id_{n-p}\oplus h)\otimes \phi'(x)\]
with $\phi' = V(\id_{n-p}\oplus \iota_{p-q})$. This morphism is independent of the choice of $h\in \eta$.
\end{prop}

This allows us to describe the face maps of the semisimpicial set $(K_\bullet F)_n$ for a $U\G$--set $F$. Recall that 
\[ d_i =  \id_i \oplus \iota_1\oplus \id_{p-i} =  (b_{1,i}\oplus \id_{p-i})G_1 \colon p \to  p+1.\] 
Then
\[ (K_pF)_n = G_n\times_{G_{n-p-1}} F_{n-p-1} \longrightarrow G_n\times_{G_{n-p}} F_{n-p} = (K_{p-1}F)_n\]
is given by 
\[ [g,x] \longmapsto [g(\id_{n-p-1}\oplus b_{1,i}\oplus \id_{p-i}),\phi'(x)].\]

\begin{ex}
The functor $F=\Hom(0,-)\colon \C \to \Set$ sends every object to a singleton. Then
\[ (K_pF)_n \cong G_n \times_{G_{n-p-1}} F_{n-p-1} = G_n/G_{n-p-1} = \Hom(p+1 ,n).\]
Which is the semisimplicial set $W_n(0,1)_n$ from \cite[Def 2.1]{RW}.
\end{ex}

\begin{ex}\label{ex:SymK}
We can now pick up  \autoref{ex:FIhomology} again. 
 Let $V\colon \FI \to \xmod{\ZZ}$. As $U\mathfrak S$ from  \autoref{ex:Symbraid} is a skeleton of $\FI$, we have
\[ (K_pV)_n \cong \ZZ\mathfrak S_n \otimes_{\ZZ \mathfrak S_{n-p-1}} V_{n-p-1} \cong \bigoplus_{f\in \Hom_{\FI}(\{n-p,\dots,n\}, \{1, \dots, n\})} V(\{1,\dots,n\}\setminus \im f)\]
by sending $g\otimes v$ to 
\[ g|_{\{1, \dots,n-p-1\}}(v) \in V(\{1,\dots,n\}\setminus g(\{n-p,\dots, n\}))\]
in the summand corresponding to
\[ f = g|_{\{n-p,\dots,n\}}.\]
\end{ex}

Functors $F\colon U\G \to \Set$ that preserve monomorphisms, such as representable functors, send all maps in $U\G$ to injective set maps. For such functors we can split $K_\bullet F$ into a disjoint union.

\begin{prop}\label{prop:injFunctors}
Let $F\colon U\G\to \Set$ preserve monomorphisms. Then the semisimplicial set $(K_\bullet F)_n$ splits disjointly
\[ (K_\bullet F)_n \cong \coprod_{x\in F_n} L^x_\bullet\]
into augmented semisimplicial sets $L_\bullet^x \subset (K_\bullet\Hom(0,-))_n$ with
\[ L_p^x= \{ \sigma \in (K_p\Hom(0,-))_n \mid x \in \im F(\sigma')\},\]
where $\sigma' = f\circ(\id_{n-p-1}\oplus \iota_{p+1})$ is a complement of $\sigma=fG_{n-p-1}$.
\end{prop}

\begin{proof}
Because $(K_{-1}F)_n = F_n$, the augmented semisimplicial set splits into 
\[ (K_\bullet F)_n \cong \coprod_{x\in F_n} L^x_\bullet\]
with $L_p^x \subset (K_pF)_n$ the preimage of $x\in F_n$ which is
\[ L_p^x = \{ [g,y]\in G_n\times_{G_{n-p-1}}F_{n-p-1} \mid g \phi'(y) =x\}\]
using  \autoref{prop:sigmap}.
Because $F$ preserves monomorphims, the set map $\phi'= F(\id_{n-p-1}\oplus \iota_{p+1})$ is injective. Thus for every $g\in G_n$ there is at most one $y\in F_{n-p-1}$ with $g\phi'(y)=x$. Therefore $L_\bullet^x$ can be included into $(K_\bullet\Hom(0,-))_n$ by sending $[g,y]$ to $gG_{n-p-1}$.

Finally, note that
\[  F(\sigma') = g\phi'.\qedhere\]
\end{proof}

\section{Central stability and finiteness properties}\label{sec:finprop}

\begin{Def}\label{Def:Cset}
We call a functor from a category $\C$ to the category of sets a \emph{$\C$--set} and write $\CSet := \Fun(\C,\Set)$. A $U\G$--set $V$ is \emph{generated in ranks $\le m$} if there is an epimorphism $P \surject V$ of $U\G$--sets where $P$ is a disjoint union 
\[ P = \coprod_{i\in I} \Hom(m_i,-)\]
of representable functors such that $m_i\le m$ for all $i\in I$.

Fix a ring $R$. We call a functor from a category $\C$ to the category of $R$--modules a \emph{$\C$--module} and write $\Cmod := \Fun(\C,\xmod{R})$. A $U\G$--module $V$ is \emph{generated in ranks $\le m$} if there is an epimorphism $P \surject V$ of $U\G$--modules where $P$ is a direct sum 
\[ P = \bigoplus_{i\in I} R\Hom(m_i,-)\]
such that $m_i\le m$ for all $i\in I$.

In both situations we call such a $P$ \emph{freely generated in ranks $\le m$}.
\end{Def}

\begin{rem}\label{rem:CEFfree}
For $U\G$--modules, a similar notion has been called ``freely generated'' by for example \cite{CEF}. The $U\G$--module $R\Hom(m,-)$ has a linear right action by $G_m$. Let $W_m$ be a not necessarily free $RG_m$--module, then
\[ R\Hom(m,-) \otimes_{RG_m} W_m\]
is a $U\G$--module. These modules are projective over $U\G$ if and only if $W_m$ is a projective $RG_m$--module. Therefore we will not call these freely generated as \cite{CEF} do in the case of $\FI$--modules.
\end{rem}

Let $V$ be a $U\G$--set or a $U\G$--module, then $V_n = V(n)$ is a sequence of $G_n$--sets or $G_n$--representations, respectively, and $\phi_n = V(\iota_1 \oplus \id_n)$ is $G_n$--equivariant.  The following lemma is a criterion when such a sequence is actually a $U\G$--set or a $U\G$--module. This criterion has been observed by Church--Ellenberg--Farb \cite[Rmk 3.3.1]{CEF} for the special case of $\FI$--modules but the argument easily generalizes as shown in \cite[Prop 4.2]{RW}.

\begin{lem}\label{lem:Cmod criterion}
Let $(V_n, \phi_n)_{n \in \mathbb N}$ be a sequence of $RG_n$--modules or $G_n$--sets (resp.) $V_n$ and $G_n$--equivariant maps $\phi_n \colon V_n \to V_{n+1}$. 

There is a unique $U\G$--module or $U\G$--set (resp.) $V$ with $V(n) = V_n$ and $V(\iota_1 \oplus \id_n)  = \phi_n$ if and only if for all $m \le n$ and  $g\in G_{n-m}$
\[(g\oplus \id_m) \circ \phi_{m,n} = \phi_{m,n},\] where $\phi_{m,n} = \phi_{n-1} \circ\dots\circ \phi_m$.
\end{lem}

We want to connect the central stability homology to generation properties.

\begin{prop}\label{prop:fg}
Let $V \colon U\G \to \D$ be a functor to $\D = \Set$ or to $\D=\xmod{R}$. Then $V$ is generated in the ranks $\le d$ if and only if the map 
\[ (\Sigma V)_n \longrightarrow V_n\]
induced by $\iota_1\in \Hom(0,1)$ is surjective for all $n> d$.
\end{prop}

\begin{proof}
We give the proof for $\D=\Set$. For $\D=\xmod{R}$ the proof is analogous after linearlizing. 

Let $p \le d < n$. Then for a representable functor 
\[  \Hom(p\oplus 1,n) \longrightarrow \Hom(p,n)\]
is given by $f \mapsto f\circ (\id_p\oplus \iota_1)$ is surjective because $G_n$ acts transitively on $\Hom(p,n)$ by postcomposition. Let there be an epimorphism 
\[ P = \coprod_{i\in I} \Hom(m_i, -) \surject V\]
such that all $m_i\le d$. The diagram
\[\xymatrix{
\Sigma P_n \ar@{->>}[r] \ar[d] & P_n \ar@{->>}[d]\\
\Sigma V_n \ar[r] & V_n
}\]
commutes because $\Sigma \to \id$ is a natural transformation between endofunctors. This implies the first implication.

Let $V$ be a functor for which
\[ (\Sigma V)_n \longrightarrow V_n\]
is surjective for all $n> d$. Let $P$ be a disjoint union
\[ P = \coprod_{i\in I} \Hom(m_i,-)\]
of representable functors with $m\le d$ together with a morphism of $U\G$--sets $P \to V$ such that
\[ P_n \surject V_n\]
is surjective for all $n\le d$. (The Yoneda Lemma lets us find such a $P$.)

Let $n >d$ and assume that $P_{n-1} \to V_{n-1}$ is surjective by induction. Then the following commutative diagram show that $P_n \to V_n$ is surjective.
\[
\begin{gathered}[b]
\xymatrix{
\displaystyle(\Sigma P)_n \cong   G_n\times_{G_{n-1}} P_{n-1} \ar@{->>}[r]\ar@{->>}[d] &\displaystyle(\Sigma V)_n \cong   G_n\times_{G_{n-1}} V_{n-1} \ar@{->>}[d]\\
P_n \ar[r] &V_n
}\\[-\dp\strutbox]
\end{gathered}\qedhere\]
\end{proof}

The previous proposition says that a $U\G$--module $V$ is generated in finite ranks if and only if $\widetilde H_{-1} V$ is stably zero, ie $\widetilde H_{-1} V_n=0$ for all $n$ large enough. We next want to generalize this concept. For this we need an additional condition.

\begin{Def}
Let $a,k\in \NN$. We define the following condition.
\begin{description}
\item[H3($N$)] $\widetilde H_i(R\Hom(0,-))_n = 0$ for all $i<N$ and all $n> k\cdot i + a$.
\end{description}
\end{Def}

\begin{rem}\label{rem:RWH3}
Note that the condition \con{LH3} in \cite[Def 2.2]{RW} implies \con{H3($\infty$)} with $a=\rk A+1$. For all of \autoref{exs:stabcat}, they prove or gather the following values from existing literature.
\begin{center}
\begin{tabular}{rp{8cm}ll}
& stability category &k&a\\\hline
(a)&$U\mathfrak S$&1&1\\
(b)&$U\GL(R)$ for a ring $R$ with stable rank $s$&2&s+1\\
(c)&$U\Sp(R)$ for a ring $R$ with unitary stable rank $s$ &2 &s+2\\
(d)&$U\AutF$&2&2\\
(e)&$U\beta$& 1&1\\
(f)&$U\Mod\Sigma$&2&2
\end{tabular}
\end{center}
\end{rem}

\begin{thm}\label{thm:res of finite type}
Assume \con{H3($N$)}. Let $V$ be a $U\G$--module and $d_0,\dots,d_N\in\ZZ$ with $d_{i+1}-d_i \ge \max(k,a)$,
then the following statements are equivalent.
\begin{enumerate}
\item There is a partial resolution
\[ P_N \to P_{N-1} \to \dots \to P_0 \to V \to 0\]
with $P_i$ that are freely generated in ranks $\le d_i$.
\item The homology
\[ \widetilde H_i(V)_n =0\]
for all $-1\le i<N$ and all $n > d_{i+1}$.
\end{enumerate}
\end{thm}

To prove this theorem we first compute the central stability homology of free $U\G$--modules. This part is quite technical and can easily be skipped upon the first read. Assuming  \autoref{prop:homology of free}, the \hyperlink{proof:thm:res of finite type}{proof} of  \autoref{thm:res of finite type} is straight forward.

To compute the central stability homology of free $U\G$--modules, we investigate $K_\bullet \Hom(m,-)$. We first introduce the concept of joins and links for augmented semisimplicial sets. We say $\rho$ is a \emph{join} $\tau * \sigma$ of the $p$--simplex $\tau$ and the $q$--simplex $\sigma$ in a semisimplicial object if
\[ \partial_0 \cdots \partial_p \rho = \sigma \quad\text{and}\quad  \partial_{p+1} \cdots \partial_{p+q+1} \rho = \tau\]
and let the left- and right-sided link be
\[ \Lk^L \sigma = \{ \tau \mid \exists \text{ a join }\tau*\sigma\} \quad\text{and}\quad \Lk^R \tau = \{ \sigma \mid \exists \text{ a join }\tau*\sigma\}.\]

In $K_\bullet \Hom(0,-)_n$ the $p$--simplices are morphisms from $ p+1 \to n$. And for $\rho \in K_{p+q+1}\Hom(0,-)_n$
\[ \partial_0 \cdots \partial_p \rho = \rho \circ (\iota_{ p+1} \oplus \id_{q+1})\quad\text{and}\quad  \partial_{p+1} \cdots \partial_{p+q+1} \rho = \rho   \circ ({\id_{p+1}} \oplus \iota_{q+1}).\]
Condition \eqref{item:complinj} of \autoref{Def:weakly complemented category} implies that $\tau*\sigma$ is unique if it exists. More generally we write $\tau*\sigma= \rho \colon l\oplus m \to n$ for $\tau\colon l \to n$ and $\sigma\colon m \to n$ if
\[  \rho\circ (\iota_l\oplus \id_m) = \sigma \quad\text{and}\quad  \rho\circ(\id_l \oplus \iota_m) = \tau.\]

\begin{lem}\label{lem:link}
Let $fG_{n-m}=\sigma\colon m \to n$, then
\[ \Lk^L \sigma = \{ \tau\colon l\to n \mid \exists \tau' \colon l\to n-m\text{ such that } \tau = f\circ (\id_{n-m} \oplus \iota_m)\circ \tau'\}\cong K_\bullet\Hom(0,-)_{n-m}\]
and
\[ \Lk^R \sigma = \{ \tau\colon l\to n \mid \exists \tau' \colon l\to n-m\text{ such that } \tau = f\circ b_{m,n-m}\circ (\iota_m \oplus \id_{n-m})\circ \tau'\}\cong K_\bullet\Hom(0,-)_{n-m}.\]
\end{lem}

\begin{proof}
Assume
\[  \tau = f\circ (\tau' \oplus \iota_m)\]
then
\[ \rho = f\circ (\tau'\oplus \id_m) = \tau*\sigma.\]
Similarly if
\[\tau = f\circ b_{m,n-m}\circ (\iota_m \oplus  \tau')\]
then
\[ \rho = f\circ b_{m,n-m}\circ (\iota_m \oplus \tau') = \tau*\sigma.\]

For the opposite direction apply  \autoref{prop:Hominj} to
\[ f^{-1} \circ \rho \colon l \oplus m \longrightarrow (n-m) \oplus m\]
and
\[ f^{-1} \circ b_{m,n-m}^{-1}\circ \rho \colon m \oplus l \longrightarrow m \oplus (n-m).\qedhere\]
\end{proof}

\begin{cor}
Let $\sigma \in K_p\Hom(0,-)_n$, then both
\[ \Lk^L \sigma \cong \Lk^R \sigma \cong K_\bullet \Hom(0,-)_{n-p-1}.\]
\end{cor}

\begin{prop}\label{prop:homology of free}
Assume \con{H3($N$)}. Then
\[ \widetilde H_i(R\Hom(m,-))_n = 0 \]
for all $i<N$ and all $n > k\cdot i +a +m$.
\end{prop}

\begin{proof}
From  \autoref{prop:injFunctors} we know that $K_p \Hom({ m},-)_n$ is isomorphic to
\[  \coprod_{\tau \in \Hom(m,n)} L_p^\tau\]
with
\[ L_p^\tau = \{ \sigma \in K_p\Hom(0,-)_n\mid \tau \in \im \Hom(m,-)(\sigma')\}\]
where $\sigma' = f\circ (\id_{n-p-1}\oplus \iota_{p+1})$ is a complement of $\sigma = fG_{n-p-1}$. Note that
\[ \im \Hom(m,-)(\sigma') \subset \Hom(m,n),\]
that is the set of those morphisms that factor through $\sigma'$, is independent of the choice of $f$. Because 
\[ \tau \in \im\Hom(m,-)(\sigma') \iff \exists \tau'\colon m \to n-p-1\text{ such that } \tau = \sigma'\circ \tau',\]
   \autoref{lem:link} gives that $\tau \in \Lk^L\sigma$, hence
 \[ L_\bullet^\tau = \Lk^R\tau \cong K_\bullet\Hom(0,-)_{n-m}.\]
This means that for every $U\G$--module $P$ freely generated in ranks $\le m$, $\widetilde H_i(P)_n = 0$ for all $-1\le i< N$ and $n>k\cdot i +a +m$.
\end{proof}

This proposition suffices to prove \autoref{thm:res of finite type}. First we will derive a corollary that the modules described in \autoref{rem:CEFfree} have vanishing central stability homology in the same range as the freely generated $U\G$--module $R\Hom(m,-)$. Note that in \autoref{thm:res of finite type}  (a) implies (b) for any partial resolution with vanishing central stability homology in the same ranges as the $P_i$ that are freely generated in ranks $\le d_i$.

\begin{cor}
Assume \con{H3($N$)}. Let $W_m$ be an $RG_m$--module. Then
\[ \widetilde H_i(R\Hom(m,-)\otimes_{RG_m}W_m)_n = 0 \]
for all $i<N$ and all $n > k\cdot i +a +m$.
\end{cor}

\begin{proof}
Let 
\[ Q_N \to Q_{N-1} \to \dots \to Q_0 \to W_m\]
be a projective resolution of $W_m$ by $RG_m$--modules. We consider the complex
\[ \widetilde C_*(R\Hom(m,-) \otimes_{RG_m} Q_*)_n \cong \widetilde C_*(R\Hom(m,-))_n \otimes_{RG_m}Q_*\]
and its two spectral sequences. The first spectral sequence
\[ E^0_{pq}=  \widetilde C_p(R\Hom(m,-))_n \otimes_{RG_m} Q_q\]
is given by
\[ E^2_{pq} = E^{\infty}_{pq} \cong \begin{cases}  \widetilde H_p\bigg(R\Hom(m,-) \otimes_{RG_m} W_m\bigg)_n & q =0\\ 0 &q>0\end{cases}\]
because 
\[ \widetilde C_p(R\Hom(m,-))_n \cong \Ind^{G_n}_{G_{n-m-p}} R\]
is a free $RG_m$--module.

The second spectral sequence
\[ E^0_{pq}=  \widetilde C_q(R\Hom(m,-))_n \otimes_{RG_m} Q_p,\]
that converges to the same limit, computes to
\[ E^1_{pq} \cong \widetilde H_q(R\Hom(m,-))_n \otimes_{RG_m} Q_p.\]
Therefore the central stability homology of $R\Hom(m,-)\otimes_{RG_m} W_m$ vanishes in the same range as $R\Hom(m,-)$.
\end{proof}

\begin{proof}[Proof of  \autoref{thm:res of finite type}]
\hypertarget{proof:thm:res of finite type}Assume  (a). Let $K_{-1}=P_{-1}:=V$ and let
\[ K_i := \ker( P_i \to P_{i-1})\]
for $i\ge 0$. Then the long exact sequence of $\widetilde H_*$ for the short exact sequence
\[ 0\longrightarrow K_i \longrightarrow P_i \longrightarrow K_{i-1} \longrightarrow0\]
implies that
\[ \widetilde H_j(K_i)_n \inject  \widetilde H_{j-1}(K_{i+1})_n\]
is injective for all $i,j\le N-1$ and all $n>kj+a+d_{i+1}$. Because $P_{i+1} \surject K_i$ for all $i\le N-1$, we know from  \autoref{prop:fg} that
\[ \widetilde H_{-1}(K_i)_n =0 \]
for all $n>d_{i+1}$. Thus
\[ \widetilde H_i(V)_n \inject H_{i-1}(K_0)_n \inject \cdots \inject \widetilde H_{-1}(K_i)_n = 0 \]
is zero for all $i\le N-1$ and $n>d_{i+1}$ because
\[ d_{i+1}-d_j \ge (i+1-j)\max(k,a) \ge k(i-j) + a\]
for $j=0,\dots,i+1$.

Now assume (b). Let by induction
\[ P_{N-1} \to \dots \to P_0 \to V \to 0\]
be a partial resolution such that $P_i$ is freely generated in  ranks $\le d_i$ for all $i\le N-1$. Let as before $K_{-1}=P_{-1}:=V$ and let
\[ K_i := \ker( P_i \to P_{i-1})\]
for $i\ge 0$. We need to prove that $K_{N-1}$ is generated in ranks $d_N$. Similar as before
\[ \widetilde H_j(K_i)_n \surject  \widetilde H_{j-1}(K_{i+1})_n\]
is surjective for all $i,j\le N-1$ and all $n>k(j-1)+a+d_{i+1}$.
Thus
\[ 0=  \widetilde H_{N-1}(V)_n  \surject \widetilde H_{N-2}(K_{0})_n \surject \cdots \surject \widetilde H_{-1}(K_{N-1})_n \]
for all 
\[n>d_{N} \ge k(N-2-j)+a+d_j.\qedhere\]
\end{proof}

\begin{Def}
We call a $U\G$--modules $V$ \emph{stably acyclic} if $\widetilde H_i(V)_n = 0$ for all $i\ge -1$ and all $n$ large enough.
\end{Def}

We want to conclude this section by giving an example that is by definition stably acyclic, but the existence of the free resolution as in \autoref{thm:res of finite type} is not \textit{a priori} clear. More generally, in \autoref{sec:poly} we find a condition \con{H4} on $U\G$ such that all $U\G$--modules with finite polynomial degree (see \autoref{Def:polynomial degree}) are stably acyclic.

\begin{cor}
Assume \con{H3($N$)}. Let $V$ be a $U\G$--module such that $V_n = 0$ for all $n>d$. Then there is a resolution
\[ P_N \to P_{N-1} \to \dots \to P_0 \to V \to 0\]
with $P_i$ that are freely generated in ranks $\le \max(k,a)\cdot i +d$.
\end{cor}

\section{Notions of central stability}\label{sec:notions}

Different notions of central stability have been used in the past. We want to clear this up for stability categories. In the following proposition we prove the equivalence of four conditions. The condition \eqref{item:presdeg} is called ``presented in degree $\le d$'' in \cite{CEF,CE}. In \cite{PS}, condition \eqref{item:Lan} is called ``central stability''. And conditions \eqref{item:Lancolim} and \eqref{item:colim} are used in \cite{CEFN}, although they phrase the colimit over an equivalent category. We will refer to this notion as in \eqref{item:presdeg} by \emph{presented in the ranks $\le d$}.

\begin{prop}\label{prop:presdeg}
Let $V$ be a $U\G$--module, then the following are equivalent.
\begin{enumerate}
\item\label{item:presdeg} $V$ is generated and presented in the ranks $\le d$, ie there are $U\G$--modules $P_0, P_1$ freely generated in the ranks $\le d$ such that
\[ P_1 \longrightarrow P_0 \longrightarrow V \longrightarrow 0\]
is exact.
\item\label{item:Lan} Let $\inc_d\colon U\G^{\le d} \inject U\G$ be the full subcategory with the objects $0, \dots, d$. The left Kan extension
\[ \Lan_{\inc_d}(V\circ \inc_d) \cong V\]
is naturally isomorphic to $V$.
\item\label{item:Lancolim} For all $n>d$ there is a natural isomorphism
\[ \colim_{\substack{m \to n\\m\le d}} V_m \cong V_n.\]
\item\label{item:colim} For all $n>d$ there is a natural isomorphism
\[ \colim_{\substack{m \to n\\m<n}} V_m \cong V_n.\]
\end{enumerate}
\end{prop}

\begin{proof}
Djament \cite[Prop 2.14]{Dj} 
proves
that \eqref{item:presdeg} and \eqref{item:Lan} are equivalent in an even more general setting. Gan--Li \cite[Thm 3.2]{GLcentral} gave a different proof with the viewpoint of graded modules over graded nonunital algebra.

The left Kan extension evaluated at $n\in \NN$ can be expressed as the colimit in \eqref{item:Lancolim} (see \cite[Cor X.3.4]{ML}):
\[ (\Lan_{\inc_d}V\circ \inc_d )_n \cong \colim_{(\inc_d\downarrow n)} V = \colim_{\substack{m \to n\\m\le d}} V_m  \]
This proves that \eqref{item:Lancolim} is equivalent to \eqref{item:Lan}.

Finally we want to prove that \eqref{item:Lancolim} and \eqref{item:colim} are equivalent. First note that
\[ \colim_{\substack{m \to n\\m<n}} V_m  =  \colim_{(\inc_{n-1} \downarrow n)} V\]
is just a different way of writing the same colimit. Further there is a natural map
\[ \colim_{(\inc_d\downarrow n)} V \longrightarrow V_n.\]
And because $(\inc_d\downarrow n)$ is a full subcategory of $(\inc_{n-1}\downarrow n)$ there is a map
\[ \colim_{(\inc_d\downarrow n)} V \longrightarrow \colim_{(\inc_{n-1} \downarrow n)} V\]
and similarly fixing a morphism $m\to n$ there is a map
\[ \colim_{(\inc_d\downarrow m)} V \longrightarrow \colim_{(\inc_{d} \downarrow n)} V.\]
Let us prove the following claim.

\textbf{Claim.} \textit{Let $n\ge d$ and assume that
\[ \colim_{(\inc_d\downarrow m)} V \longrightarrow V_m\]
is an isomorphism for all $m< n$, then
\[ \colim_{(\inc_d\downarrow n)} V \longrightarrow \colim_{(\inc_{n-1} \downarrow n)} V\]
is an isomorphism.}

\begin{proof}
\renewcommand{\qedsymbol}{$\blacksquare$}

For every morphism $f\colon m\to n$ which is an object in $(\inc_{n-1}\downarrow n)$, we get the following diagram.
\[ \xymatrix{
\colim_{(\inc_d\downarrow m)}V \ar[r]^{\cong} \ar[d]_{\alpha} & V_m \ar[d]^{\beta}\\
\colim_{(\inc_d\downarrow n)}V \ar@<.5ex>[r]^{\gamma} &\colim_{(\inc_{n-1}\downarrow n)}V\ar@<.5ex>[l]^{\delta}
}\]
Here $\delta$ is the unique map such that $\alpha = \delta\beta$. One checks also checks that $\beta = \gamma\alpha$.
This implies that
\[ \delta\gamma \alpha = \alpha \quad\text{and}\quad \gamma\delta\beta = \beta.\]
The universal property of the colimits implies that $\gamma$ and $\delta$ are inverses.
\end{proof}

We now want to finish our proof of the equivalence of \eqref{item:Lancolim} and \eqref{item:colim} by induction over $n$. For $n=d+1$ the statements are identical. Now assume that
\[ \colim_{(\inc_d\downarrow n)} V \longrightarrow \colim_{(\inc_{n-1} \downarrow n)} V\]
for all $d<m<n$. Then we can certainly assume
\[ \colim_{(\inc_d\downarrow m)} V \longrightarrow V_m\]
is an isomorphism for all $m< n$. The claim then finishes the induction step.
\end{proof}

Putman's \cite{P} original notion of central stability for $\FI$--modules is reflected in condition \eqref{item:sgn} of the following proposition. In the case of symmetric stability categories, this is equivalent to  \eqref{item:H}, which we call \emph{centrally stable} in the ranks $>d$.

\begin{prop}
Let $V$ be a $U\G$--module, then the following are equivalent.
\begin{enumerate}
\item \label{item:H} $\widetilde H_0(V)_n = \widetilde H_{-1}(V)_n = 0$ for all $n >d$.
\item \label{item:nosgn} $V_n \cong \coeq\big(\Ind^{G_n}_{G_{n-2}} V_{n-2} \rightrightarrows \Ind^{G_n}_{G_{n-1}} V_{n-1}\big)$ for all $n>d$.
\end{enumerate}
If $\G$ is symmetric, then the following condition is equivalent to the first two.
\begin{enumerate}
\setcounter{enumi}{2}
\item \label{item:sgn}$V_n \cong \coeq\big(\Ind^{G_n}_{G_{n-2}\times \mathfrak S_2} V_{n-2} \otimes \A_2 \rightrightarrows \Ind^{G_n}_{G_{n-1}} V_{n-1}\big)$ for all $n>d$, where $\mathfrak S_2\subset G_2$ and $\A_2$ is its sign representation.
\end{enumerate}
\end{prop}

\begin{proof}
Conditions \eqref{item:H} and \eqref{item:nosgn} are equivalent because \eqref{item:nosgn} is equivalent to the sequence
\[ \Ind^{G_n}_{G_{n-2}} V_{n-2}  \longrightarrow \Ind^{G_n}_{G_{n-1}} V_{n-1} \longrightarrow V_n \longrightarrow 0\]
being exact, which is part of the central stability complex.

To see that \eqref{item:nosgn} and \eqref{item:sgn} are equivalent, we have to check the image that the images in $\Ind^{G_n}_{G_{n-1}} V_{n-1}$ coincide. Let $t\in \mathfrak S_2$ be the transposition, then both maps are given by
\[ v \mapsto (1-t) \otimes \phi'(v)\]
for $v\in V_{n-2}$. This completes the proof.
\end{proof}

With \autoref{thm:res of finite type} we can connect these two notions. Note that the assumption of the theorem can be expressed as follows.

\begin{description}
\item[H3($1$)] $b_{1,1}\oplus \id_{n-2}$ and $1\times G_{n-1}$ generate $G_n$ for all $n>a$.
\end{description}

\begin{cor}\label{cor:pres central}
Assume \con{H3($1$)}. Let $V$ be a $U\G$--module. Then the following are equivalent.
\begin{enumerate}
\item $V$ is generated in ranks $\le d$ and presented in ranks $\le d+a$.
\item $V$ is generated in ranks $\le d$ and centrally stable in the ranks $>d+a$.
\end{enumerate}
\end{cor}

\begin{cor}
\begin{enumerate}
\item A $U\mathfrak S$--module is generated in ranks $\le d$ and presented in ranks $\le d+1$ if and only if it is generated in ranks $\le d$ and centrally stable in the ranks $>d+1$.
\item Let $R$ be a ring with stable rank $s$. A $U\GL(R)$--module is generated in ranks $\le d$ and presented in ranks $\le d+s+1$ if and only if it is generated in ranks $\le d$ and centrally stable in the ranks $>d+s+1$.
\item Let $R$ be a ring with unitary stable rank $s$. A $U\Sp(R)$--module is generated in ranks $\le d$ and presented in ranks $\le d+s+2$ if and only if it is generated in ranks $\le d$ and centrally stable in the ranks $>d+s+2$.
\item A $U\AutF$--module is generated in ranks $\le d$ and presented in ranks $\le d+2$ if and only if it is generated in ranks $\le d$ and centrally stable in the ranks $>d+2$.
\item A $U\beta$--module is generated in ranks $\le d$ and presented in ranks $\le d+1$ if and only if it is generated in ranks $\le d$ and centrally stable in the ranks $>d+1$.
\item A $U\Mod\Sigma$--module is generated in ranks $\le d$ and presented in ranks $\le d+2$ if and only if it is generated in ranks $\le d$ and centrally stable in the ranks $>d+2$.
\end{enumerate}
\end{cor}


\section{Polynomial degree}\label{sec:poly}

As we defined $S$ as the precomposition of $-\oplus 1$, for the next definition we will use $T$ defined as the precomposition of $1 \oplus -$. The following definition is an adaptation of van der Kallen's \cite{vdK} degree of a coefficient system. Our version is almost identical with \cite[Def 4.10]{RW}.


\begin{Def}\label{Def:polynomial degree}
Let $V$ be a $U\G$--module then we define the $U\G$--modules 
\[\ker V = \ker(V \to TV)\text{ and }\coker V = \coker(V\to TV).\] 
We say $V$ has \emph{polynomial degree $-\infty$ in the ranks $> d$} if $V_n=0$ for all $n> d$. For $r\ge0$ we say $V$ has \emph{polynomial degree $\le r$ in the ranks $> d$} if $(\ker V)_n = 0$ for all $n > d$ and $\coker V$ has polynomial degree $\le r-1$ in the ranks $> d$. 
\end{Def}

\begin{rem}
Assuming that the dimension of $V_n$ is finite for all $n> d$, $V$ having polynomial degree $\le r$ in the ranks $> d$ implies that the dimension grows polynomially of degree $\le r$ in the ranks $> d$.
\end{rem}

\begin{lem}\label{lem:finite degree}
\begin{enumerate}
\item Let
\[ 0\longrightarrow V' \longrightarrow V \longrightarrow V'' \longrightarrow 0\]
be an exact sequence of $U\G$--modules in ranks $> d$, i.e.
\[ 0\longrightarrow V'_n \longrightarrow V_n \longrightarrow V''_n \longrightarrow 0\]
is a short exact sequence for all $n> d$. Then if two of the three $U\G$--modules $V',V,V''$ have polynomial degrees $\le r$ in ranks $> d$, then so does the third.
\item Assume the base ring $R$ is a field. If $V$ and $W$ are $U\G$--modules have polynomial degrees $\le r$ and $\le s$ in ranks $> d$, respectively, then  $V\otimes W$ has polynomial degree $\le r+s$ in ranks $> d$, where $V\otimes W$ is defined by $(V\otimes W)_n = V_n \otimes W_n$. 
\end{enumerate}
\end{lem}

\begin{proof}
\begin{enumerate}
\item From the snake lemma we get an exact sequence
\[ 0\longrightarrow \ker V'_n\longrightarrow \ker V_n\longrightarrow \ker V''_n \longrightarrow \coker V'_n \longrightarrow \coker V_n \longrightarrow \coker V''_n \longrightarrow 0\]
for all $n> d$. If $V'$ has polynomial degree $r$ in ranks $> d$,  
\[ \xymatrix{
\ker V''_n  \ar[r]\ar@{=}[d]& \coker V'_n\\
\ker^2 V''_n \ar[r]& \ker \coker V'_n =  0 \ar[u]
}\]
for all $n> d$. Hence this sequence actually splits into two short exact sequences. If $V$ and $V''$ have finite polynomial degree, then the above sequence still splits into two short exact sequences. The assertion is shown by induction on the degree.
\item If $-\infty \in \{r,s\}$, also $V\otimes W$ has polynomial degree $-\infty$ in ranks $> d$. Thus $\ker(V\otimes W)_n = 0$ for all large enough $n$. For an induction argument we consider the following commutative diagram.
\[\xymatrix{
&0\ar[d]&0\ar[d]&0\ar[d]\\
&V\otimes \ker W \ar[d] \ar[r]&\ker(V\otimes W) \ar[d] \ar[r]&\ker V \otimes TW\ar[d]\\
&V\otimes W \ar[d] \ar@{=}[r]& V\otimes W \ar[d] \ar[r]& V \otimes TW\ar[d]\\
&V\otimes TW \ar[d] \ar[r]&T(V\otimes W)\ar[d] \ar@{=}[r] & TV \otimes TW\ar[d]\\
& V\otimes \coker W \ar[d] \ar[r]& \coker (V\otimes W) \ar[d] \ar[r]& \coker V \otimes TW\ar[d] \ar[r]&0\\
&0&0&0
}\]
Clearly the three columns are exact. By chasing this diagram, we can prove that the top and the bottom row are also exact. Even more,
\[ 0\longrightarrow V_n\otimes \coker W_n \longrightarrow \coker (V\otimes W)_n \longrightarrow \coker V_n\otimes TW_n \longrightarrow 0\]
is a short exact sequence for all  $n> d$ because $V_n\otimes TW_n \to TV_n \otimes TW_n$ is injective for all $n> d$. By induction  $\coker(V\otimes W)$ has polynomial degree $\le r+s-1$ in ranks $> d$ using (a).\qedhere
\end{enumerate}
\end{proof}

The following lemma is essentially \cite[Lem 3.11]{PS}.

\begin{lem}\label{lem:zeromap}
Let $V$ be a $U\G$--module then the map $\id_n\oplus \iota_1\colon n \to n \oplus 1$ induces the zero map on $\widetilde H_{-1}V$. This implies that every morphism of $U\G$ that is not an isomorphism induces the zero map on $\widetilde H_{-1}V$.
\end{lem}

\begin{proof}
Let
\[ V_n \to \Sigma V_{n+1} \cong RG_{n+1}\otimes_{RG_n} V_n \]
be given by the inclusion $v\mapsto 1\otimes v$. Then the following diagram commutes by the calculations from  \autoref{prop:SigmaMod}.
\[\xymatrix{
& V_n \ar[ld]\ar[r] \ar[d] &  \widetilde H_{-1}V_n \ar[r]\ar[d]&0\\
\Sigma V_{n+1} \ar[r] & V_{n+1} \ar[r] &\widetilde H_{-1}(V)_{n+1}  \ar[r]&0
}\]
The exactness of the lower row implies that $V_n \to \widetilde H_{-1}V_{n+1}$ must be the zero map. Thus surjectivity of $V_n\to \widetilde H_{-1}V_n$ implies the assertion.
\end{proof}

\begin{prop}\label{prop:fgcoker}
Let $V$ be a $U\G$--module and $d\in \NN$. If $ \coker V$ is generated in ranks $\le d-1$, then $ V$ is generated in ranks $\le d$.
\end{prop}

\begin{proof}
We consider the following diagram.
\[\xymatrix{
V \ar[r] \ar@{->>}[d] & TV \ar[r] \ar@{->>}[d]& \coker V \ar[r] \ar@{->>}[d]& 0\\
\widetilde H_{-1}V \ar[r] & T\widetilde H_{-1}V \ar[r] & \coker \widetilde H_{-1}V \ar[r] & 0
}\]
The map $\widetilde H_{-1}V \to T\widetilde H_{-1}V$ is the zero map because of \autoref{lem:zeromap}. Thus there is an epimorphism
\[ \xymatrix{ \coker V  \ar@{->>}[r] & T \widetilde H_{-1} V}.\]
This implies that $T \widetilde H_{-1} V$  is generated in ranks $\le d-1$. Using  \autoref{lem:zeromap} again, we infer that $\widetilde H_{-1}V_{n+1} = T\widetilde H_{-1}V_n = 0$ for all $n> d$. Therefore $V$ is generated in ranks $\le d$ by  \autoref{prop:fg}.
\end{proof}

\begin{Def}
Let $\ell,b\ge 1$. We define the following condition.
\begin{description}
\item[H4($N$)] $\widetilde H_i(\coker R\Hom(m,-))_n = 0$ for all $i<N$ and all $n > \ell \cdot( i +  m )+ b$.
\end{description}
\end{Def}

\begin{rem}\label{rem:shiftfree}
Knowing that $\coker R\Hom(m,-)$ is stably acyclic for all $m\in \NN$, one can also follow that $T^r R\Hom(m,-)$ is stably acyclic for all $m,r\in \NN$.
\end{rem}


\begin{prop}\label{prop:H(coker V)}
Let $N\in \mathbb N$.  Assume \con{H3($N$)} and \con{H4($N$)} with $b\ge \max(k,a)$. Let $d_0,\dots, d_N\le -1$ with $d_i\ge \max(\ell \cdot d_{i-1} + b, d_{i-1}+b)$.
If $\ker V_n =0$ for all $n>d_1$ and $\widetilde H_i(\coker V)_n=0$ for all $i\le N-1$ and all $n> d_{i+1}$, then $\widetilde H_i(V)_n = 0$ for all $i\le N-1$ and $n> d_{i+1}+1$.
\end{prop}

\begin{proof}
The case $N= 0$ is proved in  \autoref{prop:fgcoker}.

Let $N\ge 1$. By induction
\[ \widetilde H_i(V)_n = 0 \]
for all $n> d_{i+1}+1$ for $i<N-1$.  \autoref{thm:res of finite type} implies the existence of a resolution
\[P_{N-1} \longrightarrow \dots \longrightarrow P_0 \longrightarrow V \longrightarrow 0\]
by $U\G$--modules $P_i$ that are freely generated in ranks $\le d_i+1$, because $d_i - d_{i-1} \ge b \ge \max(k,a)$. 
If we can prove that $K_{N-1} = \ker( P_{N-1} \to P_{N-2})$ is  generated in ranks $\le d_N+1$, we are done. Let 
\[ K_j = \ker( P_j \to P_{j-1})\]
with the convention $P_{-1} = K_{-1} = V$. Because $T$ is exact, the snake lemma gives us the exact sequence
\[ 0= \ker P_{i+1}\longrightarrow \ker K_{i} \longrightarrow  \coker K_{i+1} \longrightarrow \coker P_{i+1} \longrightarrow \coker K_{i} \longrightarrow 0.\]

If $i\ge0$, $\ker K_{i}$ is a submodule of $\ker P_{i}=0$. 
Therefore there is a long exact sequence 
\[ \cdots \to  \widetilde H_{N-i-2}(\coker K_{i})_n \to  \widetilde H_{N-i-3}(\coker K_{i+1})_n  \to  \widetilde H_{N-i-3}(\coker P_{i+1})_n \to \cdots.\] 
By the  assumption on $\widetilde H_{N-i-3}(\coker P_{i+1})_n$, we see that
\[ \widetilde H_{N-i-2}(\coker K_i)_n \surject  \widetilde H_{N-i-3}(\coker K_{i+1})_n\]
is surjective for all $i\le N-2$ and all $n>\ell\cdot (d_{i+1}+1+{N-i-3})+b$.

If $i=-1$, ie $\ker K_i = \ker V$, this method only infers
\[ \widetilde H_{N-1}(\coker V)_n \surject  \widetilde H_{N-2}(\coker K_0/ \ker V)_n\]
is surjective for all $n>\ell\cdot (d_{0}+N-1)+b$. Because $\ker V_n = 0 $ for all $n > d_1$,
\[ (\widetilde C_{N-2}\ker V)_n \cong \Ind_{G_{n-N+1}}^{G_n} \ker V_{n-N+1} = 0\]
for all $n >d_1+N-1$.
Therefore in the same range
\[ \widetilde H_{N-2}(\coker K_0)_n \inject \widetilde H_{N-2}(\coker K_0/ \ker V)_n\]
is injective.

Thus the assumption that $\widetilde H_{N-1}(\coker V)_n= 0 $ for all  $n> d_{N}$ implies 
\[ \widetilde H_{-1}(\coker K_{N-1})_n=0 \] 
in the same range because
\[ d_N \ge \ell^{N-i-1}\cdot d_{i+1} + (1+\dots+ \ell^{N-i-2})b \ge  \ell\cdot (d_{i+1}+N-i-2)+b \]
and
\[ d_N \ge \ell^{N-1}d_{1} + (1+\dots+ \ell^{N-2})b \ge d_1 + N-1 \ge \ell \cdot d_0 +b +N -1.\]

By  \autoref{prop:fgcoker}, $K_{N-1}$ is generated in ranks $\le d_N +1$.
\end{proof}

A direct consequence of this proposition is \autoref{thmB}:

\begin{cor}\label{quantThmB}
Assume \con{H3($\infty$)} and \con{H4($\infty$)} with $b\ge \max(k,a)$. Let $r\ge1$, $d\ge-1$ and let $V$ be of  polynomial degree $\le r$ in ranks $>d$, then $\widetilde H_iV_n = 0$ for all $i\ge -1$ and $n> \ell^{i+1}(d+r) + (\ell^i + \dots + 1)b+1$.
\end{cor}

\begin{proof}
For $r= 0$, $\ker V_n =\coker V_n = 0$ for all $n > d$. Then $\widetilde H_i(\coker V)_n =0$ for all $n> d+i+1$.
Let  $d_0^0 = d$ and for $i>0$ let
\[ d^0_i =  \ell^i \cdot d + (\ell^{i-1} + \dots + 1)b\]
if $d\ge 0$ and
\[ d^0_i =  -\ell^{i-1}+ (\ell^{i-1} + \dots + 1)b\]
if $d=-1$. Then $d_i^0\ge \max(\ell \cdot d_{i-1}^0 +b,d_{i-1}^0+b)$, $\ker V_n =0$ for all
\[ n > d^0_1 \ge d,\]
and $\widetilde H_i(\coker V)_n = 0$ for all 
\[n>d^0_{i+1}  \ge   d+i+1.\] 
Therefore by  \autoref{prop:H(coker V)},
\[ \widetilde H_iV_n = 0\]
for all $n > d^0_{i+1} +1$. This proves the case $r=0$.

Let $r>0$ and define
\[d^r_i =\ell^{i}(d_0^0+r) + (\ell^{i-1} + \dots + 1)b .\]
By induction $\widetilde H_i\coker V_n = 0$ for all
\[ n > d^r_{i+1}  \ge d^{r-1}_{i+1}+1.\]
We also have that
\[ d^r_i \ge \ell d^r_{i-1} +b\]
and 
\[ \ker V_n= 0\]
for all
\[ n>  d_1^r \ge  d. \]
Thus we can apply  \autoref{prop:H(coker V)} to get
\[ \widetilde H_iV_n = 0\]
for all $n > d^r_{i+1} +1$. This proves the case $r>0$.
\end{proof}

From \con{H3($N$)} and a long exact sequence we get that
\[ \widetilde H_i(T\Hom(m,-))_n \surject \widetilde H_i(\coker\Hom(m,-))_n\]
is surjective for all $i\le N-1$ and all $n>k\cdot (i-1) + a + m$. With the next proposition we will express $K_\bullet T\Hom(m,-)_n$ in form of the semisimplicial sets $K_\bullet\Hom(0,-)_n$ that are analyzed in \cite{RW}. 


\begin{prop}
The augmented semisimplicial set $K_\bullet T\Hom( m,-)_n$ is isomorphic to the disjoint union of augmented semisimplicial sets
\[ \coprod_{\tau \in \Hom(m,n+1)} \Lk^R\tau \cap K_\bullet\Hom(0,-)_n\]
with the embedding of $K_\bullet\Hom(0,-)_n \subset K_\bullet\Hom(0,-)_{n+1}$ by $\sigma \mapsto (\iota_1 \oplus \id_n) \circ \sigma$.
\end{prop}

\begin{proof}
From \autoref{prop:injFunctors} we know that $K_p T\Hom({ m},-)_n$ is isomorphic to
\[  \coprod_{\tau \in \Hom(m,n+1)} L_p^\tau\]
with
\[ L_p^\tau = \{ \sigma \in K_p\Hom(0,-)_n\mid \tau \in \im T\Hom(m,-)(\sigma')\}\]
where $\sigma' = f\circ (\id_{n-p-1}\oplus \iota_{p+1})$ is a complement of $\sigma = fG_{n-p-1}$. The equation
\[ \id_1 \oplus \sigma' = \Big((\iota_1 \oplus \id_n)\circ \sigma\Big)'\]
and 
\[ \tau \in \im T\Hom(m,-)(\sigma') \iff \exists \tau'\colon m \to 1 \oplus (n-p-1)\text{ such that } \tau = (\id_1\oplus \sigma')\circ \tau',\]
together with  \autoref{lem:link} gives that $\tau \in \Lk^L((\iota_1\oplus \id_n)\circ \sigma)$, hence
 \[ L_\bullet^\tau = \Lk^R\tau \cap K_\bullet\Hom(0,-)_{n}.\qedhere\]
\end{proof}

\begin{ex}\label{ex:Ubeta}
Here we want to provide an example of a stability category that satisfies \con{H3($\infty$)} but does not even satisfy \con{H4($0$)}. We will explain why $T\Hom_{U\beta}(1,-)$ is not generated in finite ranks. Let $\sigma = b\beta_{n} \in T\Hom_{U\beta}(1,-)_n = \Hom_{U\beta}(1,1 \oplus n)$ which is the coset that is represented by a braid $b\in \beta_{n+1}$. Now 
\[ \phi(\sigma) = (\id_1 \oplus \iota_1 \oplus \id_{n}) \circ \sigma = (b_{1,1}^{-1} \oplus \id_n) \circ (\id_1 \oplus b) \beta_{n+1},\]
which is depicted in the following diagram.
\[  (b_{1,1}^{-1} \oplus \id_n) \circ (\id_1 \oplus b) = \quad 
\begin{tikzpicture}[x=.5cm,y=-1.5cm,baseline=-2cm]

\node[v]  at (0,0) {};
\node[v]  at (1,0) {};
\node[v]  at (2,0) {};
\node at (3.5,0) {$\cdots$};
\node[v]  at (5,0) {};
\node[v]  at (6,0) {};
\node[yshift=12pt] at (6,0) {\scalebox{.8}{$1$}};

\node at (3.5,.5) {$b$};

\node[v]  at (0,1) {};
\node[v]  at (1,1) {};
\node[v]  at (2,1) {};
\node at (3.5,1) {$\cdots$};
\node[v]  at (5,1) {};
\node[v]  at (6,1) {};

\draw (.7,0) to (6.3,0) to (6.3,1) to (.7,1) to (.7,0);

\draw [decorate,decoration={brace,amplitude=4pt,raise=4pt},yshift=0pt]
(0,0) -- (5,0) node [black,midway,yshift=20pt]  {\scalebox{.8}{$\beta_{n+1}$}};
\node[yshift=12pt] at (2.5,0) {\scalebox{.8}{$\lefttorightarrow$}};

\draw[very thick] (0,0) to (0,1);

\node at (3,1.25) {$\circ$};

\draw[very thick] (1,1.5) to[out=270, in=90] (0,2.5);
\draw[white , double=black , very thick , double distance =1.3pt] (0,1.5) to[out=270, in=90] (1,2.5);
\draw[very thick] (2,1.5) to[out=270, in=90] (2,2.5);
\draw[very thick] (5,1.5) to[out=270, in=90] (5,2.5);
\draw[very thick] (6,1.5) to[out=270, in=90] (6,2.5);

\node[v]  at (0,1.5) {};
\node[v]  at (1,1.5) {};
\node[v]  at (2,1.5) {};
\node at (3.5,1.5) {$\cdots$};
\node[v]  at (5,1.5) {};
\node[v]  at (6,1.5) {};

\node[v] (a) at (0,2.5) {};
\node[v]  at (1,2.5) {};
\node[v]  at (2,2.5) {};
\node at (3.5,2.5) {$\cdots$};
\node[v]  at (5,2.5) {};
\node[v]  at (6,2.5) {};

\draw [decorate,decoration={brace,amplitude=4pt,raise=4pt},yshift=0pt]
(6,2.5) -- (1,2.5) node [black,midway,yshift=-12pt] {\scalebox{.8}{$\righttoleftarrow$}};
\node[yshift=-20pt] at (3.5,2.5) {\scalebox{.8}{$\beta_{n+1}$}};
\end{tikzpicture}
\]
Our claim, that $T\Hom_{U\beta}(1,-)$ is not generated in finite ranks, amounts to showing
\[ (1\times \beta_{n+1} )  (b_{1,1}^{-1} \oplus \id_n)  (1\times \beta_{n+1}) (\beta_{n+1} \times 1) \neq \beta_{n+2}.\]
This follows because the braid
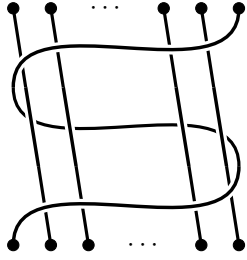
\[\begin{tikzpicture}[x=.5cm,y=-1.5cm,baseline=-2cm]
\draw[very thick] (0,0) to (.33,.7);
\draw[very thick] (1,0) to (1.33,.7);
\draw[very thick] (4,0) to (4.33,.7);
\draw[very thick] (5,0) to (5.33,.7);
\draw[white , double=black , very thick , double distance =1.3pt] (6,0) to[out=270, in=90] (0,.7);
\draw[very thick] (0,.7) to[out=270, in=90] (6,1.4);
\draw[white , double=black , very thick , double distance =1.3pt] (.33,.7) to (.67,1.4);
\draw[white , double=black , very thick , double distance =1.3pt] (1.33,.7) to (1.67,1.4);
\draw[white , double=black , very thick , double distance =1.3pt] (4.33,.7) to (4.67,1.4);
\draw[white , double=black , very thick , double distance =1.3pt] (5.33,.7) to (5.67,1.4);
\draw[very thick] (.67,1.4) to (1,2.1);
\draw[very thick] (1.67,1.4) to (2,2.1);
\draw[very thick] (4.67,1.4) to (5,2.1);
\draw[very thick] (5.67,1.4) to (6,2.1);
\draw[white , double=black , very thick , double distance =1.3pt] (6,1.4) to[out=270, in=90] (0,2.1);

\node[v]  at (0,0) {};
\node[v]  at (1,0) {};
\node[v]  at (4,0) {};
\node at (2.5,0) {$\cdots$};
\node[v]  at (5,0) {};
\node[v]  at (6,0) {};

\node[v]  at (0,2.1) {};
\node[v]  at (1,2.1) {};
\node[v]  at (2,2.1) {};
\node at (3.5,2.1) {$\cdots$};
\node[v]  at (5,2.1) {};
\node[v]  at (6,2.1) {};
\end{tikzpicture}
\]
is not contained in the LHS.
\end{ex}

\section{Short exact sequences}\label{sec:stabSES}

In this section will deal with short exact sequences of groups and stability groupoids. We restate some well-known properties of group homology regarding these short exact sequences in the language of modules over stability categories. Most of the arguments of this section have already appeared in \cite{PS} but we hope that they become more accessible in the language of this paper.

Let
\[ 1 \longrightarrow N \longrightarrow G \longrightarrow Q \longrightarrow 1\]
be a short exact sequence of groups. We want to describe the group homology of $N$ as a $Q$--module. For this let us first review how an automorphism $\phi \in \Aut(N)$ acts on $H_i(N;M)$ for some $RN$--module $M$ (given an $RN$--homomorphism $\psi\colon M \to \Res_\phi M$, ie $\psi(n\cdot m) = \phi(n) \cdot \psi(m)$.) Let $E_*N$ be a free (right) $RN$--resolution of the trivial representation $R$. Let $\xi\colon E_*N \to \Res_\phi E_*N$ be an $RN$--homomorphism. (This can always be found and any two are chain homotopic.) Then $\phi$ induces the map
\[ E_*N \otimes_N M \stackrel{\xi \otimes \psi}{\longrightarrow} \Res_\phi E_*N\otimes_N \Res_\phi M.\]
Note that the RHS is canonically isomorphic to $E_*N \otimes_N M$ because $\phi$ is a bijection.

Let $\phi$ be an inner automorphism of $N$, say conjugation by $n\in N$. Then $\xi(x) = xn^{-1}$ and $\psi(m) = nm$ fulfill the above requirements. But then $\xi \otimes \psi$ is in fact the identity on $E_*N \otimes_N M$. 

Let $E_*G$ be a free (right) $RG$--resolution of the trivial representation $R$. Assume $M$ is the restriction of an $RG$--module and $\phi$ the conjugation by an element $g\in G$. ($N$ is normal in $G$.) Then $\xi(x) = xg^{-1}$ and $\psi(m) = gm$ fulfill the above requirements as before. Note that $\xi\otimes \psi$ is not the identity on $E_*G \otimes_N M$. Summarizing, we have seen that $E_*G\otimes_N M$ is an $RQ$--module by
\[ Ng \cdot (x\otimes m) = xg^{-1} \otimes gm.\]
This  induces an $RQ$--module structure on $H_*(N;M)$ for every $RG$--module $M$.

Now we generalize this concept. 

\noindent\fbox{\parbox[b]{\textwidth}{\textit{In this section, we deal with multiple stability groupoids $\mathcal N = (N_n)_{n\in\NN}, \G= (G_n)_{n\in\NN}, \mathcal Q=(Q_n)_{n\in\NN}$.}

\textit{For the most time we assume that $\G$ and $\mathcal Q$ are braided, in which their stability categories $U\G$ and $U\mathcal Q$ are monoidal and prebraided, whence it makes sense to take of their central stability complex and homology, which we denote by $\widetilde C_*^\G, \widetilde H_*^\G$ and $\widetilde C_*^{\mathcal Q}, \widetilde H_*^{\mathcal Q}$, respectively.}}}

\begin{Def}\label{Def:stability SES}
Let $F\colon \mathcal N \to \mathcal G$ and $F'\colon \mathcal G \to \mathcal Q$ be homomorphisms of stability groupoids. We call this data a \emph{stability SES} if
\[ 1 \longrightarrow N_n  \longrightarrow G_n \longrightarrow  Q_n \longrightarrow 1 \]
is a short exact sequence for all $n\in \NN$.
\end{Def}

\begin{prop}\label{prop:criterion stab SES}
Let $\mathcal G, \mathcal Q$ be stability groupoids and $F'\colon \mathcal G \to \mathcal Q$ a homomorphism of stability groupoids and $F'_n \colon G_n \to Q_n$ is surjective for every $n\in \mathbb N$. Then there is a stability groupoid $\mathcal N$ and a homomorphism of stability groupoids $F\colon \mathcal N \to \mathcal G$ such that
\[ \mathcal N  \stackrel{F}{\longrightarrow}  \mathcal G  \stackrel{F'}{\longrightarrow} \mathcal Q \]
is a stability SES.
\end{prop}

\begin{proof}
Let $N_n$ be the kernel of $F'_n$ and $F_n$ the embedding of $N_n$ in $G_n$. Let $\mathcal N$ be the groupoid formed by these groups and $F$ the functor given by all $F_n$'s. Because
\[ \ker(F'_{m}\times F'_{n}) = N_{m} \times N_{n}\]
we get an injective group homomorphism
\[ N_m \times N_{n} \inject N_{m+n}.\]
This proves that $\mathcal N$ is a stability groupoid. All other assertions are immediate.
\end{proof}

\begin{lem}
Let
\[ 1 \longrightarrow \mathcal N \longrightarrow \mathcal G \longrightarrow \mathcal Q \longrightarrow 1 \]
be a stability SES and $V$ a $U{\mathcal G}$--module. Then for every $i \ge0$ there is a $U{\mathcal Q}$--module $W$ with
\[ W_n \cong H_i (N_n; \Res_{N_n}^{G_n} V_n).\]
We denote this $U{\mathcal Q}$--module by $H_i(\mathcal N; V)$ and abbreviate $H_i(\mathcal N)$ if $V=R\Hom(0,-)$. 
\end{lem}

\begin{proof}
We will use \autoref{lem:Cmod criterion} for this proof.

We already know that for every $i\ge 0$ we have a sequence of $RQ_n$--representations. Let
\[ \phi_n \colon H_i(N_n;V_n) \longrightarrow H_i(N_{n+1};V_{n+1})\]
be induced by $\iota_1 \oplus \id_n$. On the chain complex level this is
\[ E_*G_{n+1} \otimes_{N_n} V_n \longrightarrow E_*G_{n+1} \otimes_{N_{n+1}} V_{n+1}\]
induced by $\phi = V(\iota_1\oplus \id_n) \colon V_n \to V_{n+1}$. 

That means that
\[ \phi\colon H_i(N_m; V_m) \longrightarrow H_i(N_n,V_n)\]
is induced by the map
\[ E_*G_n \otimes_{N_m} V_m \stackrel{\id\otimes \phi}{\longrightarrow} E_*G_n \otimes_{N_n} V_n.\]
Because $x\mapsto xg'$ for some $g' \in G_{n-m}$ is a $G_m$--automorphism of the free resolution $E_*G_n$, it is homotopic to the identity. By functoriality the homotopy descends to $E_*G_n \otimes_{N_m} V_m$. Therefore $x\otimes m \mapsto xg' \otimes m$ induces the identity on $H_i(N_m;V_m)$. By the following commutative diagram, the action of $G_{n-m}$ on $\phi(H_i(N_m;V_m))$ is trivial. This finishes the proof by application of \autoref{lem:Cmod criterion}.

\[\begin{gathered}[b]\xymatrix{
E_*G_n \otimes_{N_m} V_m \ar[r]^{\id\otimes \phi}\ar[d]^{\cdot (g')^{-1}\otimes \id} & E_*G_n \otimes_{N_n} V_n\ar[d]^{\cdot (g')^{-1}\otimes g'\cdot} \\
E_*G_n \otimes_{N_m} V_m \ar[r]^{\id\otimes \phi} & E_*G_n \otimes_{N_n} V_n
}\\[-\dp\strutbox]
\end{gathered}\qedhere\]
\end{proof}

The central technical tool of this section is the following spectral sequence.

\begin{prop}\label{prop:spectralsequence}
Let
\[ 1 \longrightarrow \mathcal N \longrightarrow \mathcal G \longrightarrow \mathcal Q \longrightarrow 1 \]
be a stability SES. Assume that $\G$, $\mathcal Q$, and $\G\to\mathcal Q$ are braided. Let $V$ be a $U\G$--module.
Then the two spectral sequences constructed from the double complex
\[ E_* G_n \otimes_{RN_n} \widetilde C_*^{\mathcal G} (V)_n\]
are
\[E^1_{pq} \cong E_p G_n \otimes_{RN_n} \widetilde H_q^{\mathcal G} (V)_n\]
and 
\[ E^2_{pq} \cong \widetilde H_p^{\mathcal Q}( H_q(\mathcal N; V))_n. \]
\end{prop}

\begin{proof}
The first spectral sequence is given by
\[ E^0_{pq} = E_p G_n \otimes_{RN_n} \widetilde C_q^{\mathcal G} (V)_n.\]
Therefore
\[E^1_{pq} \cong E_p G_n \otimes_{RN_n} \widetilde H_q^{\mathcal G} (V)_n.\]

The other spectral sequence is given by
\begin{multline*}
E^1_{pq} \cong H_q(N_n; \widetilde C_p^{\mathcal G}(V)_n) = H_q(N_n; RG_n \otimes_{RG_{n-(p+1)}} V_{n-(p+1)}) \\= H_q( E_*G_n \otimes_{RN_n} RG_n \otimes_{RG_{n-(p+1)}} V_{n-(p+1)}).
\end{multline*}
The differential $d^1$ is induced by the map
\begin{align*}
E_*G_n \otimes_{RN_n} RG_n \otimes_{RG_{n-(p+1)}} V_{n-(p+1)} &\longrightarrow E_*G_n \otimes_{RN_n} RG_n \otimes_{RG_{n-p}} V_{n-p}\\
x \otimes g \otimes v & \longmapsto \sum (-1)^i x \otimes gg_i \otimes \phi'(v)
\end{align*}
with $g_i = \id_{n-p-1}\oplus b_{1,i}\oplus \id_{p-i}$ and $\phi' = V(\id_{n-p}\oplus \iota_1)$.

Then there is a $Q_n$--equivariant isomorphism
\[ E_*G_n \otimes_{RN_n} RG_n \otimes_{RG_{n-(p+1)}} V_{n-(p+1)}  \cong RQ_n \otimes_{RQ_{n-(p+1)}} \big(E_*G_n \otimes_{RN_{n-(p+1)}} V_{n-(p+1)}\big)\]
given by
\begin{align*}
x\otimes g \otimes v &\longmapsto N_ng \otimes (xg \otimes v)\\
xg^{-1} \otimes g \otimes v& \longmapsfrom N_ng \otimes (x \otimes v).
\end{align*}
Therefore
\[ E_{pq}^1 \cong \widetilde C^{\mathcal Q}_p (H_i(\mathcal N;V))_n.\]

Because the following diagram commutes, $d^1$ coincides with the differential of $\widetilde C^{\mathcal Q}_*$. 

\[\xymatrix{
E_*G_n \otimes_{RN_n} RG_n \otimes_{RG_{n-(p+1)}} V_{n-(p+1)}  \ar[r]\ar[dd] & RQ_n \otimes_{RQ_{n-(p+1)}} \big(E_*G_n \otimes_{N_{n-(p+1)}} V_{n-(p+1)}\big)\ar[d]\\
& RQ_n \otimes_{RQ_{n-(p+1)}} \big(E_*G_n \otimes_{N_{n-(p+1)}} V_{n-(p+1)}\big) \ar[d] \\
E_*G_n \otimes_{RN_n} RG_n \otimes_{RG_{n-p}} V_{n-p} \ar[r]\ & RQ_n \otimes_{RQ_{n-(p+1)}} \big(E_*G_n \otimes_{N_{n-p}} V_{n-p}\big)
}\]
is given by
\[\begin{gathered}[b]\xymatrix{
x\otimes g \otimes v \ar@{|->}[r]\ar@{|->}[dd] & N_ng \otimes (xg \otimes v)\ar@{|->}[d]\\
&N_ng \otimes (xgg_i \otimes v) \ar@{|->}[d] \\
x\otimes gg_i \otimes \phi'(v) \ar@{|->}[r]\ & N_ngg_i \otimes (xgg_i \otimes \phi'(v)).
}\\[-\dp\strutbox]\end{gathered}\qedhere\]
\end{proof}

The idea for this spectral sequence can be found in \cite{PS}. The following immediate consequence illustrates how it can be used.

\begin{cor}
Let
\[ 1 \longrightarrow \mathcal N \longrightarrow \mathcal G \longrightarrow \mathcal Q \longrightarrow 1 \]
be a stability SES. Assume that $\mathcal G$, $\mathcal Q$, and $\G\to\mathcal Q$ are braided. 
For each $n$,  there is a spectral sequence 
\[ E^2_{p,q} \cong \widetilde H^{\mathcal Q}_p( H_q(\mathcal N))_n \]
that converges to zero for $ p+q \le  \frac{n-a-1}{k}$ if $U\G$ satisfies  \con{H3($\infty$)}. 
\end{cor}

\begin{proof}
The two spectral sequences in \autoref{prop:spectralsequence} converge to the same limit. The first will converge to zero for $n\ge k\cdot q +a+1$ by \con{H3($\infty$)} if we set $V$ to be $R\Hom(0,-)$. In particular, the diagonals $p+q \le \frac{n-a-1}k$ converge to zero. Thus the same is true for the second spectral sequence. 
\end{proof}

 The proof of \autoref{thmC} uses the same methods as Putman and Sam in \cite[Theorem 5.13]{PS}.

\begin{proof}[Proof of \autoref{thmC}]
We may use the spectral sequences from  \autoref{prop:spectralsequence}. The first spectral sequence implies that both converge stably to zero, because $V$ is stably acyclic. We will now prove the result by induction on $i$. 

For $i=0$ note that because 
\[\doublequot{N_n}{G_n}{G_{n-m}} \cong \quot{Q_n }{Q_{n-m}}\]
we get
\[ H_0(\mathcal N;R\Hom_{U\G}(m,-)) \cong R\Hom_{U{\mathcal Q}}(m,-).\]
Because $H_0$ is right exact it follows that $H_0(\mathcal N;V)$ is a finitely generated $U{\mathcal Q}$--module.

Let $i>0$ and assume $H_q(\mathcal N;V)$ is finitely generated for all $q<i$. From the Noetherian condition we get that $H_q(\mathcal N;V)$ is stably acyclic for all $q<i$. Now using the second spectral sequence from  \autoref{prop:spectralsequence}, we infer $E^2_{-1i} = \widetilde H_{-1}(H_i(\mathcal N;V))$ is stably zero, because the only incoming and outgoing differentials are all stably zero.  \autoref{prop:fg} and the assumption that $H_i(\mathcal N;V)_n$ is a finitely generated $R$--module for all $n$ implies that $H_i(\mathcal N;V)$ is finitely generated.
\end{proof}


\section{Quillen's argument revisited}\label{sec:Quillen}

In this section we will solely prove \autoref{thmD} from the introduction.

\begin{proof}[Proof of \autoref{thmD}]
We consider the stability SES
\[ 1  \longrightarrow \G \stackrel{\id}{\longrightarrow}  \G \longrightarrow \NN \longrightarrow 1\]
where $\NN$ is the category whose objects are the natural numbers and whose morphisms are only the identity maps.  In \autoref{prop:spectralsequence}, we constructed two spectral sequences for this situation. From the first
\[ E^1_{pq} =  E_pG_{n+1} \otimes_{G_{n+1}} \widetilde H^\G_q V_{n+1}\]
which is zero for $n+1> k\cdot q +a$, we see that both spectral sequences converge to zero when $p+q \le  \frac{ n- a}k$. The other spectral sequence is given by
\[ E^1_{pq} \cong \widetilde C^\NN_p(H_q(\G;   V))_{n+1} \cong H_q(G_{n+1-(p+1)};V_{n+1-(p+1)}).\]
The differentials are easy to understand:
\[ H_q(G_{n+1-(p+1)};V_{n+1-(p+1)}) \to H_q(G_{n+1-p};V_{n+1-p})\]
is zero if $p$ is odd and it is $\phi_*$ if $p$ is even. In particular $E^1_{0,i} \to E^1_{-1,i}$ is the stabilization map
\[ \phi_*\colon H_i(G_n; V_n) \longrightarrow H_i(G_{n+1};V_{n+1}).\]

Assume $n\ge ki +a -1$, we want to prove that $\phi_*$ is surjective. We know that $E^\infty_{-1,i} = 0$ if $i-1\le \frac{n-a}k$, in particular when $n\ge ki+ a-1$. We want to use induction to show that $E^2_{pq} = 0$ when $p+q=i$ and $q<i$. This would imply that $E^2_{-1,i}$ already vanishes and $\phi_*$ is surjective. If $p$ is even, to show that $E^2_{pq}=0$ it suffices to show that
\[ E^2_{p+1,q} \cong H_q(G_{n-p-1};V_{n-p-1}) \longrightarrow E^2_{pq} \cong H_q(G_{n-p};V_{n-p})\]
is surjective. By induction this is the case if $n-p-1 \ge kq +a -1$. If $p$ is odd, it suffices to show that
\[ E^2_{pq} \cong H_q(G_{n-p};V_{n-p}) \longrightarrow E^2_{p-1,q} \cong H_q(G_{n-p+1};V_{n-p+1})\]
is injective. By induction this is true if $n-p \ge kq + a $, which is the same condition. We know that
\[ p+kq \le i + (k-1)(i-1) = ki  - k +1 \le n -a \]
which is what we need.

Assume $n\ge ki+a$, we want to prove that $\phi_*$ is injective. We know that  $E^\infty_{0,i}  =  0$ for $i\ge \frac{n-a}k$, in particular when $n\ge ki +a$. Again we prove $E^2_{pq} = 0$ when $p+q=i+1$ and $q<i$ by induction, which implies that $E^2_{0,i} = 0$ and $\phi_*$ is injective. We already computed that $E^2_{pq}$ vanishes when $n-p \ge kq +a$. And we calculate
\[ p+kq \le i+1 + (k-1)(i-1) = ki -k +2 \le n -a.\qedhere\]
\end{proof}



\bibliographystyle{halpha}
\bibliography{repstab}

\end{document}